\documentclass[a4paper,reqno]{amsart}
\usepackage[normalem]{ulem}
\usepackage{amsmath, amssymb, amsthm,mathtools}
\usepackage[english]{babel}
\usepackage{enumerate}
\usepackage[all]{xy}
\usepackage{graphicx}
\usepackage{mathtools}
\usepackage{mathrsfs}
\usepackage{hyperref}
\usepackage{verbatim}
\usepackage{color}
\usepackage{tikz-cd}
\tikzcdset{arrow style=tikz, diagrams={>=stealth}}

\newcommand{\blu}[1]{{\color{blue}{#1}}}

\allowdisplaybreaks
\DeclareMathOperator{\Hom}{Hom}

\newcommand*{\id}{\textup{id}}

\numberwithin{equation}{section}

\theoremstyle{plain}

\newtheorem{thm}{Theorem}[section]
\newtheorem{lem}[thm]{Lemma}
\newtheorem{prop}[thm]{Proposition}
 \newtheorem{cor}[thm]{Corollary}
\newtheorem{defi}[thm]{Definition}

\theoremstyle{remark}

\newtheorem{rem}[thm]{Remark}

\numberwithin{equation}{section}
\newcommand\inv{^{-1}}

\newcommand{\ot}{\otimes}

\newcommand{\beq}{\begin{equation}}
\newcommand{\eeq}{\end{equation}}

\newcommand{\cL}{\mathcal{L}}
\newcommand{\cR}{\mathcal{R}}

\newcommand{\BB}{\overline{B}}

\newcommand{\M}{\mathcal{M}}

\newcommand{\one}[1]{{#1}{}_{\scriptscriptstyle{(1)}}}
\newcommand{\two}[1]{{#1}{}_{\scriptscriptstyle{(2)}}}
\newcommand{\three}[1]{{#1}{}_{\scriptscriptstyle{(3)}}}

\newcommand{\tuno}[1]{{#1}{}{}^{\scriptscriptstyle{<1>}}}
\newcommand{\tdue}[1]{{#1}{}{}^{\scriptscriptstyle{<2>}}}
\newcommand{\un}{{}^{\scriptscriptstyle{<1>}}}
\newcommand{\du}{{}^{\scriptscriptstyle{<2>}}}

\newcommand{\yi}[1]{{#1}{}{}^{\scriptscriptstyle{[1]}}}
\newcommand{\er}[1]{{#1}{}{}^{\scriptscriptstyle{[2]}}}

\newcommand{\p}{{}_{\scriptscriptstyle{\hat{+}}}}
\newcommand{\np}{{}_{\scriptscriptstyle{\hat{[+]}}}}

\newcommand{\m}{{}_{\scriptscriptstyle{\hat{-}}}}
\newcommand{\nm}{{}_{\scriptscriptstyle{\hat{[-]}}}}

\newcommand{\dcrosss}[2]{\prescript{#1}{}\bowtie_#2}
\newcommand{\dcrossto}{\dcrosss \tau\omega}
\newcommand{\dcrossok}{\dcrosss \omega\kappa}
\newcommand{\dcrosstt}{{\,\prescript{\tau}{}\bowtie_\tau}}
\newcommand{\dcrosstk}{{\,\prescript{\tau}{}\bowtie}_\kappa}
\newcommand{\dcrosskk}{{\,\prescript{\kappa}{}\bowtie}_\kappa}
\newcommand{\dcrosskt}{{\,\prescript{\kappa}{}\bowtie}_\tau}

\newcommand{\lbiprod}{{>\!\!\!\triangleleft\kern-.33em\cdot}}
\newcommand{\rbiprod}{{\cdot\kern-.33em\triangleright\!\!\!<}}

\newcommand{\z}{{}_{\scriptscriptstyle{(0)}}}
\newcommand{\rz}{{}_{\scriptscriptstyle{[0]}}}
\renewcommand{\o}{{}_{\scriptscriptstyle{(1)}}}
\newcommand{\ro}{{}_{\scriptscriptstyle{[1]}}}
\newcommand{\mo}{{}_{\scriptscriptstyle{(-1)}}}
\newcommand{\rmo}{{}_{\scriptscriptstyle{[-1]}}}
\renewcommand{\t}{{}_{\scriptscriptstyle{(2)}}}
\newcommand{\rt}{{}_{\scriptscriptstyle{[2]}}}
\newcommand{\mt}{{}_{\scriptscriptstyle{(-2)}}}

\renewcommand{\th}{{}_{\scriptscriptstyle{(3)}}}
\newcommand{\rth}{{}_{\scriptscriptstyle{[3]}}}
\newcommand{\mth}{{}_{\scriptscriptstyle{(-3)}}}

\newcommand{\fo}{{}_{\scriptscriptstyle{(4)}}}

\newcommand{\di}{{\diamond_{B}}}
\newcommand{\la}{{\triangleright}}
\newcommand{\ra}{{\triangleleft}}
\newcommand{\bla}{{\blacktriangleright}}
\newcommand{\bra}{{\blacktriangleleft}}

\DeclareMathOperator{\tens}{\otimes}

\newcommand{\CM}{\mathcal{M}}

\newcommand{\can}{{\rm can}}

\begin{document}

\author{Xiao Han, Peter Schauenburg}
\address[]{\textit{Xiao Han},
Queen Mary University of London.
}
\email{x.h.han@qmul.ac.uk}

\address[]{\textit{Peter Schauenburg},
Universit\'e de Bourgogne.
}
\email{peter.schauenburg@u-bourgogne.fr}
%


\keywords{Hopf algebroid, bialgebroid, quantum group, Hopf Galois extensions}

\title{Hopf Galois extensions of Hopf algebroids}
%

\begin{abstract}
 We study Hopf Galois extensions of Hopf algebroids as a generalization of the theory for Hopf algebras. More precisely, we introduce  (skew-)regular comodules and generalize the structure theorem for relative Hopf modules. Also, we show that if $N\subseteq P$ is a left $\mathcal{L}$-Galois extension and $\Gamma$ is a 2-cocycle of $\mathcal{L}$, then for the twisted comodule algebra ${}_{\Gamma}P$, $N\subseteq{}_{\Gamma}P$ is a left Hopf Galois extension of the twisted Hopf algebroid $\mathcal{L}^{\Gamma}$. We study twisted Drinfeld doubles of Hopf algebroids as examples for the Drinfeld twist theory. Finally, we introduce cleft extension and $\sigma$-twisted crossed products of Hopf algebroids. Moreover, we show the equivalence of cleft extensions, $\sigma$-twisted crossed products, and Hopf Galois extensions with normal basis properties, which generalize the theory of cleft extensions of Hopf algebras.
 \end{abstract}

\maketitle

\section{Introduction}

Similarly to the relation between Hopf algebras and groups, Hopf algebroids can be viewed as a quantization of Groupoids. There are several different definitions of Hopf algebroids, in this paper, we will mainly focus on the Hopf algebroids introduced in \cite{schau1} and \cite{schau3}. Namely, a Hopf algebroid is a left Hopf algebroid and an anti-left Hopf algebroid.

For a left bialgebroid $\mathcal{L}$ over a noncommutative ring $B$, a left $\mathcal{L}$-Galois extension can be  defined in a natural way. However, there is no obvious definition of a  Hopf Galois extension on the right-hand side; the direct analog of the canonical map in the definition of "Galois" for the ordinary Hopf case is not well defined in the Hopf algebroid case. To deal with this problem we rather study anti-right $\mathcal{L}$-Galois extensions. As a regular left
comodule of itself, $\mathcal{L}$ is a left Hopf algebroid, if and only if $\BB\subseteq \mathcal{L}$ is a left $\BB$-Galois extension. Similarly, as a regular right $\cL$-comodule, $\cL$ is an anti-left Hopf algebroid if and only if  $B\subseteq \mathcal{L}$ is an anti-right $\cL$-Galois extension. For a classical Hopf algebra $H$ and a left $H$-comodule $P$, one can always make $P$ into a right $H$ comodule, by defining the right coaction $\delta:p\to p\z\ot S^{-1}(p\mo)$. In order to generalize this property to Hopf algebroid, we introduce the notion of a skew regular left $\cL$-comodule $P$, which means the map $\phi: \cL\ot_{B}P\to \cL\di P$,  $X\ot p\mapsto p\mo X\ot  p\z$ is bijective. We show that if $\cL$ is an anti-left Hopf algebroid and $P$ is a left $\cL$-comodule, then $P$ is skew regular. Moreover, we show that any left bialgebroid which admits a faithfully flat left Hopf Galois extension is a left Hopf algebroid; this generalizes the main result in \cite{schau2}. Similarly, we show any left bialgebroid admits a faithfully flat skew regular left Hopf Galois extension is an anti-right Hopf algebroid. Next, as a generalization of the structure theorem in \cite{schneider} we study the structure theorem of Hopf modules of Hopf algebroids, namely, for any faithfully flat left $\cL$-Galois extension $N\subseteq P$, we have the category equivalence ${}_{N}\M\simeq {}_{P}^{\cL}\M$. In addition, if $P$ is skew regular, we also have $\M_{N}\simeq {}^{\cL}\M_{P}$. We also study all these  facts on anti-right Hopf Galois extensions of Hopf algebroids.

 As the main result of \cite{HM22}, if $\cL$ is a (anti-)left Hopf algebroid and $\Gamma$ is a 2-cocycle on $\cL$, then the Drinfeld twist bialgebroid $\cL^{\Gamma}$ is also a (anti-)left Hopf algebroid. Motivated by this result, we study the Drinfeld twist in relation with Hopf Galois extensions and show that if $N\subseteq P$ is a (skew regular) left Hopf $\cL$-extension, then $N\subseteq {}_{\Gamma}P$ is also a (skew regular) left Hopf Galois extension, where ${}_{\Gamma}P$ is the left hand twisted comodule algebra of $\cL^{\Gamma}$. This generalizes the result for the Hopf algebra case \cite{MS}. We find a nice example based on  generalized Drinfeld Doubles. Recall that \cite{schau1}, given a shew pairing $\tau$ between two left $B$-bialgebroids $\cL$ and $\Pi$, one can construct a new bialgbroid $\Pi\dcrosstt \cL$. If there is another skew pairing $\kappa$, one can similarly define a $B^{e}$-ring $\Pi\dcrosstk \cL$ which is a left $\Pi\dcrosstt \cL$-comodule and right $\Pi\dcrosskk \cL$-comodule induced by a one-side cocycle twist respectively. Also, $\Pi\dcrosstt \cL$ is a 2-cocycle twist of $\Pi\dcrosskk \cL$. Moreover, if both $\cL$ and $\Pi$ are Hopf algebroid, then $\BB\subseteq\Pi\dcrosstk \cL$ is a left Hopf Galois extension and $B\subseteq\Pi\dcrosskk \cL$ is an anti-right Hopf Galois extension.

In \cite{BB}, cleft extensions of full Hopf algebroids (a bialgebroid with antipode and compatible left and right bialgebroid structure) have already been studied. As full Hopf algebroids are a much more restrictive structure, we are motivated to generalize the theory for cleft extensions of Hopf algebroids. We know that for a Hopf algebra $H$, $H$ is a cleft extension of itself with the cleaving map given by the identity map and the inverse of the cleaving map given by the antipode. However, since there is no antipode for a Hopf algebroid $\cL$, $\cL$ is  not even a cleft extension of itself. So to define a appropriate notion of left $\cL$-cleft extension $P$, we should not require a $\cL$-colinear map $\gamma$ which is convolution invertible. Instead, we  require the map $j:\cL\ot_{\BB}P\to \cL\di P$, $X\ot p\mapsto X\o\ot \gamma(X\t)p$ to be bijective. It is not hard to see this  recovers the usual definition in the  case of Hopf algebras since the inverse of $j$ can be given in terms of $\gamma^{-1}$, namely, $j^{-1}(h)=h\o\ot \gamma^{-1}(h\t)$.
Although the requirement of a $\cL$-cleft extensions seems weaker than the one for classical Hopf algebras, we can still show the equivalence of cleft extensions and Hopf Galois extensions with the normal basis property.  We can also prove the equivalence of these two notions with a suitable generalization of crossed products, but the generalization of the classical Hopf algebra case poses additional problems.
To wit, a left cleft extension $N\subset P$ over a Hopf algebra $H$ is isomorphic to a crossed product $H\#_\sigma N$ defined with respect to a generalized right action of $H$ on $N$ and two-cocycle $\sigma$. If $H$ is replaced by a Hopf algebroid $\cL$ however, there is no well-defined notion of (generalized) right action of $\cL$ on an algebra $N$, e.~g.\ already no notion of right $\cL$-module algebra.
We solve this by using a more intricate notion of crossed product using a weak left action of the co-opposite $\cL^{\operatorname{cop}}$ as well as a two cocycle on $\cL^{\operatorname{cop}}$. In the Hopf algebra case, the new construction translates to the classical construction using the antipode. Finally, we will also study the equivalence classes of cleft extensions and twisted crossed products.

\subsection*{Acknowledgements} Xiao Han was supported by the European Union's Horizon 2020 research and innovation program under the Marie Sklodowska-Curie grant agreement No 101027463 at the early stage of the project. Xiao Han was supported by COST Action CA21109 for Short-Term Scientific Mission. Xiao Han is grateful to the Institut de Mathematiques de Bourgogne
(Dijon) for hospitality.

\section{Basic algebraic preliminaries} \label{sec2}

In this section, we will recall some basic definitions and notation. Let $B$ be an unital algebra over a field $k$. We denote the opposite algebra by $\BB$ and let $B\to \BB$, $b\mapsto\Bar{b}$ for any $b\in B$ be the obvious $k$-algebra antiisomorphism. Define $B^{e}:=B\ot \BB$, so $B$ and $\BB$ are obvious subalgebras of $B^{e}$. Let $M, N$ be $B^{e}$-bimodules. We define
\begin{align*}
    M\di N:=\int_{b} {}_{\Bar{b}}M\ot {}_{b}N:=&M\ot N/\langle \Bar{b}m\ot n-m\ot bn|b\in B, m\in M, n\in N\rangle\\
    M\ot_{B} N:=\int_{b} M_{b}\ot {}_{b}N:=&M\ot N/\langle mb\ot n-m\ot bn|b\in B, m\in M, n\in N\rangle\\
    M\ot_{\BB} N:=\int_{b} M_{\Bar{b}}\ot {}_{\Bar{b}}N:=&M\ot N/\langle m\Bar{b}\ot n-m\ot \Bar{b}n|b\in B, m\in M, n\in N\rangle\\
\end{align*}
For convenience, we also define $N\ot^{B}M=\int_{b} {}_{b}N\ot M_{b}$ and  $N\ot^{\BB}M=\int_{b} {}_{\Bar{b}}N\ot M_{\Bar{b}}$. Moreover, we define
\begin{align*}
    \int^{b}M_{\Bar{b}}\ot N_{b}:=\{\sum_{i}m_{i}\ot n_{i}\in M\ot N\quad |\quad m_{i}\Bar{b}\ot n_{i}=m_{i}\ot n_{i}b, \forall b\in B\}.
\end{align*}

The symbol $\int^{b}$ and $\int^{c}$ commute, also $\int_{b}$ and $\int_{c}$ commute. However, in general, the symbol $\int^{b}$ and $\int_{c}$ doesn't commute. For any $\BB$-bimodule $M$ and any $B$-bimodule $N$, we also define
\begin{align*}
    M\times_{B}N:=\int^{a}\int_{b} {_{\Bar{b}}M_{\Bar{a}}}\ot {_bN_a}.
\end{align*}
$M\times_{B}N$ is called Takeuchi product of $M$ and $N$. If $P$ is a $B^{e}$-bimodule, then $P\times_{B}N$ is a $B$-bimodule with $B$ acting on $P$. Similarly, $M\times_{B}P$ is a $\BB$-bimodule with $\BB$ acting on $P$. If both $M$ and $N$ are $B^{e}$-bimodule, then $M\times N$ is also a $B^{e}$-bimodule. However, the product $\times_{B}$ is neither associative and unital on the category of $B^{e}$-bimodules. For any $M,N, P\in {}_{B^{e}}\M_{B^{e}}$, we can define
\begin{align*}
    M\times_{B}P\times_{B}N:=\int^{a,b}\int_{c,d} {}_{\Bar{c}}M_{\Bar{a}}\ot {}_{c,\Bar{d}}P_{a, \Bar{b}}\ot {}_{d}N_{b},
\end{align*}
where $\int^{a,b}:=\int^{a}\int^{b}$ and $\int_{c,d}:=\int_{c}\int_{d}$. There are maps
\begin{align*}
    &\alpha:(M\times_{B}P)\times_{B} N\to M\times_{B}P\times_{B}N,\quad m\ot p\ot n\mapsto m\ot p\ot n,\\
    &\alpha':M\times_{B}(P\times N)\to M\times_{B}P\times_{B}N,\quad m\ot p\ot n\mapsto m\ot p\ot n.\\
\end{align*}
Notice that neither $\alpha$ nor $\alpha'$ are isomorphisms  in general. For the rest of the paper, we will assume that all $B$-module and $\BB$-module structures are faithfully flat. In particular this impliest that $\alpha$ and $\alpha'$ are in fact isomorphisms. 

\subsection{Hopf algebroids}

Here we recall the basic definitions (cf. \cite{BW}, \cite{Boehm}). Let $B$ be a unital algebra over a field $k$.
A {\em $B$-ring}   means a unital algebra in the monoidal category ${}_B\CM_B$ of $B$-bimodules. Likewise,  a  {\em $B$-coring} is a coalgebra in ${}_B\CM_B$. Morphisms of $B$-(co)rings are defined to be morphisms of (co)algebras, but in the category ${}_B\CM_B$.

Specifying a unital $B$-ring $\cL$ is equivalent to specifying a unital algebra $\cL$ (over $k$) and an algebra map $\eta:B\to \cL$. Left and right multiplication in $\cL$ pull back to left and right $B$-actions as a bimodule (so $bXc=\eta(b)X\eta(c)$ for all $b,c\in B$ and $X\in \cL$) and the product descends to the product $\mu_B:\cL\tens_B\cL\to \cL$ with $\eta$ the unit map. Conversely, given
$\mu_B$ we can pull back to an associative product on $\cL$ with unit $\eta(1)$.

Now suppose that $s:B\to \cL$ and $t:\BB\to \cL$ are algebra maps with images that commute. Then $\eta(b\tens c)=s(b)t(c)$ is an algebra map $\eta: B^e\to \cL$, where $B^e=B\tens \BB$, and is equivalent to making $\cL$ a $B^e$-ring. The left $B^e$-action part of this is equivalent to a $B$-bimodule structure
\begin{equation}\label{eq:rbgd.bimod}
b.X.c=b\Bar{c}X:= s(b) t(c)X
\end{equation}
for all $b,c\in B$ and $X\in \cL$. Similarly, the right $B^{e}$-action part of this is equivalent to another $B$-bimodule structure
\begin{equation}\label{eq:rbgd.bimod1}
c^{.}X^{.}b=Xb\Bar{c}:= Xs(b) t(c).
\end{equation}

If $\cL$ and $\mathcal R$ are two $B^e$-rings, then the Takeuchi product $\cL\times_B\mathcal R$ is an algebra with multiplication given by $(\ell\otimes r)(\ell'\otimes r')=\ell\ell' \otimes rr'$ for $\ell\ot r,\ell'\ot r'\in \cL\times_B\mathcal R$.

\begin{defi}\label{def:left.bgd} Let $B$ be a unital algebra. A left $B$-bialgebroid (or left bialgebroid over $B$) is an algebra $\cL$ and (`source' and `target') commuting algebra maps $s:B\to \cL$ and $t:\BB\to \cL$ (thus making $\cL$ a $B^e$-ring), such that
\begin{itemize}
\item[(i)]  There are two left $B^{e}$-module maps, the coproduct $\Delta:\cL\to \cL\times_{B} \cL$ and counit $\varepsilon:\cL\to B$ satisfying
\begin{align*}
&\alpha\circ(\Delta\times_{B}\id)\circ\Delta=\alpha'\circ(\id\times_{B}\Delta)\circ\Delta:\cL\to \cL\times_{B}\cL\times_{B}\cL\\    &\varepsilon(X\o)X\t=\overline{\varepsilon(X\t)}X\o=X,
\end{align*}
for any $X\in\cL$, where we use the sumless Sweedler index to denote the coproduct, namely, $\Delta(X)=X\o\ot X\t$.
\item[(ii)] The coproduct is an algebra map. The counit $\varepsilon$ is a left character in the following sense:
\begin{equation*}\varepsilon(1_{\cL})=1_{B}, \quad \varepsilon(X\varepsilon(Y))=\varepsilon(XY)=\varepsilon(X\overline{\varepsilon(Y)})\end{equation*}
for all $X,Y\in \cL$ and $a\in B$.
\end{itemize}
\end{defi}

\begin{rem}
   Condition (i) imply $\cL$ is a $B$-coring with the $B$-bimodule structure given by (\ref{eq:rbgd.bimod}). By condition (ii), we can define an algebra map $\hat{\varepsilon}: \cL\to \operatorname{End}(B)$ by $\hat{\varepsilon}(X)(b):=\varepsilon(X b)$. Moreover, the image of $\alpha\circ(\Delta\times_{B}\id)\circ\Delta=\alpha'\circ(\id\times_{B}\Delta)\circ\Delta$ could be also denoted by the Sweedler index, namely, $\alpha\circ(\Delta\times_{B}\id)\circ\Delta(X)=\alpha'\circ(\id\times_{B}\Delta)\circ\Delta(X)=X\o\ot X\t\ot X\th\in \cL\times_{B}\cL\times_{B}\cL$. Unlike the bialgebra case, the elements $X\o\o\ot X\o\t\ot X\t$, $X\o\ot X\t\o\ot X\t\t$ and $X\o\ot X\t\ot X\th$ are not the same, since they belong to three different vector spaces, even if in many interesting cases both $\alpha$ and $\alpha'$ are isomorphisms.
   Morphisms between left $B$-bialgebroids are $B$-coring maps which are also $B^e$-ring map maps.
\end{rem}


\begin{defi}\label{defHopf}
A left bialgebroid $\cL$ over $B$ is a left Hopf algebroid (\cite{schau1}, Thm and Def 3.5.) if
\[\lambda: \cL\ot_{\BB}\cL\to \cL\di\cL,\quad
    \lambda(X\ot Y)=\one{X}\ot \two{X}Y\]
is invertible.\\
Similarly, A left bialgebroid $\cL$ over $B$ is a anti-left Hopf algebroid if
\[\mu: \cL\ot^{B}\cL\to \cL\di \cL,\quad
    \mu(X\ot Y)=\one{Y}X\ot \two{Y}\]
is invertible. A left $B$-bialgebroid is a $B$-Hopf algebroid if it is a left Hopf algebroid and anti-left Hopf algebroid.
\end{defi}
In the following we will always use the balanced tensor product explained as above. If $B=k$ then this reduces to the map $\cL\tens\cL\to \cL\tens\cL$ given by $h\tens g\mapsto h\o\tens h\t g$ which for a usual bialgebra has an inverse, namely $h\tens g\mapsto h\o\tens (Sh\t)g$ if and only if there is an antipode. We adopt the shorthand
\begin{equation}\label{X+-} X_{+}\ot_{\BB}X_{-}:=\lambda^{-1}(X\di 1)\end{equation}
\begin{equation}\label{X[+][-]} X_{[-]}\ot^{B}X_{[+]}:=\mu^{-1}(1\di X).\end{equation}
So for a Hopf algebra $H$, $h_{+}\ot h_{-}=h\o\ot S(h\t)$ and $h_{[-]}\ot h_{[+]}= S^{-1}(h\o)\ot h\t$ for any $h\in H$.
We recall from \cite[Prop.~3.7]{schau1} that for a left Hopf algebroid, and any $X, Y\in \cL$ and $a,a',b,b'\in B$,
\begin{align}
    \one{X_{+}}\di{}\two{X_{+}}X_{-}&=X\di{}1\label{equ. inverse lamda 1};\\
    \one{X}{}_{+}\ot_{\BB}\one{X}{}_{-}\two{X}&=X\ot_{\BB}1\label{equ. inverse lamda 2};\\
    (XY)_{+}\ot_{\BB}(XY)_{-}&=X_{+}Y_{+}\ot_{\BB}Y_{-}X_{-}\label{equ. inverse lamda 3};\\
    1_{+}\ot_{\BB}1_{-}&=1\ot_{\BB}1\label{equ. inverse lamda 4};\\
    \one{X_{+}}\di{}\two{X_{+}}\ot_{\BB}X_{-}&=\one{X}\di{}\two{X}{}_{+}\ot_{\BB}\two{X}{}_{-}\label{equ. inverse lamda 5};\\
    X_{+}\ot\one{X_{-}}\ot{}\two{X_{-}}&=X_{++}\ot X_{-}\ot{}X_{+-}\in\int^{a,b}\int_{c,d}{}_{\Bar{a}}\cL_{\Bar{c}}\ot {}_{\Bar{d}}\cL_{\Bar{b}}\ot {}_{\Bar{c},d}\cL_{b,\Bar{a}}
    \label{equ. inverse lamda 6};\\
    X&=X_{+}\overline{\varepsilon(X_{-})}\label{equ. inverse lamda 7};\\
    X_{+}X_{-}&=\varepsilon(X)\label{equ. inverse lamda 8};\\
    (a\Bar{a'}Xb\Bar{b'})_{+}\ot_{\BB}(a\Bar{a'}Xb\Bar{b'})_{-}&=aX_{+}b\ot_{\BB}b'X_{-}a'\label{equ. inverse lamda 9};\\
    \Bar{b}X_{+}\ot_{\BB}X_{-}&=X_{+}\ot_{\BB}X_{-}\Bar{b}\label{equ. inverse lamda 10}.
\end{align}
Similarly, for an anti-left Hopf algebroid, from \cite{BS} we have
\begin{align}
    \one{X_{[+]}}X_{[-]}\di{}\two{X_{[+]}}&=1\di{}X\label{equ. inverse mu 1};\\
    \two{X}{}_{[-]}\one{X}\ot^{B}\two{X}{}_{[+]}&=1\ot^{B}X\label{equ. inverse mu 2};\\
    (XY)_{[-]}\ot^{B}(XY)_{[+]}&=Y_{[-]}X_{[-]}\ot^{B}X_{[+]}Y_{[+]}\label{equ. inverse mu 3};\\
    1_{[-]}\ot^{B}1_{[+]}&=1\ot^{B}1\label{equ. inverse mu 4};\\
    X_{[-]}\ot^{B}\one{X_{[+]}}\di{}\two{X_{[+]}}&=\one{X}{}_{[-]}\ot^{B}\one{X}{}_{[+]}\di{}\two{X}\label{equ. inverse mu 5};\\
    \one{X_{[-]}}\ot{}\two{X_{[-]}}\ot X_{[+]}&=X_{[+][-]}\ot{}X_{[-]}\ot X_{[+][+]}\in \int^{a,b}\int_{c,d}{}_{c,\Bar{d}}\cL_{a,\Bar{b}}\ot {}_{d}\cL_{b}\ot {}_{a}\cL_{c}\label{equ. inverse mu 6};\\
    X&=X_{[+]}\varepsilon(X_{[-]})\label{equ. inverse mu 7};\\
    X_{[+]}X_{[-]}&=\overline{\varepsilon(X)}\label{equ. inverse mu 8};\\
    (a\Bar{a'}Xb\Bar{b'})_{[-]}\ot^{B}(a\Bar{a'}Xb\Bar{b'})_{[+]}&=\Bar{b}X_{[-]}\Bar{a}\ot^{B}\Bar{a'}X_{[+]}\Bar{b'}\label{equ. inverse mu 9};\\
    X_{[-]}b\ot^{B}X_{-}&=X_{[+]}\ot^{B}b X_{-}\label{equ. inverse mu 10}.
\end{align}

\begin{prop}\label{prop. anti left and left Galois maps}
    If $\cL$ is a left and anti-left Hopf algebroid, then we have
    \begin{align*}
        X_{[+]}\ot X_{[-]+}\ot X_{[-]-}&=X\t{}_{[+]}\ot X\t{}_{[-]}\ot X\o\in \cL\ot_{B}\cL\ot_{\BB}\cL;\\
        X_{+}\ot X_{-[+]}\ot X_{-[-]}&=X\o{}_{+}\ot X\o{}_{-}\ot X\t\in \cL\ot_{\BB}\cL\ot_{B}\cL;\\
        X_{[+]+}\ot X_{[-]}\ot X_{[+]-}&=X_{+[+]}\ot X_{+[-]}\ot X_{-}\in \int^{c,d}\int_{a,b} {}_{c, \Bar{d}}\cL_{\Bar{a}, b}\ot {}_{b}\cL_{c}\ot {}_{\Bar{a}}\cL_{\Bar{d}},
    \end{align*}
   for any $X\in\cL$.
\end{prop}
\begin{proof}
    The first equality can be proved by comparing the result of $\id_{\cL}\ot \lambda $ on both sides. The second
 equality can be proved by comparing the result of $\id_{\cL}\ot \mu$ on both sides. The third
 equality can be proved by comparing the result of $\mu\ot \id_{\cL}$ on both sides.
\end{proof}

\section{Hopf Galois extensions of Hopf algebroids}

In this section, we will introduce the notions of a (skew)  regular comodule of left bialgebroid and Hopf Galois extensions.

\subsection{Comodules of left bialgebroids}

By \cite{schau1}, we have the definition of left comodules of left algebroids.

\begin{defi}
    Let $\cL$ be a left bialgebroid over $B$. A left $\cL$ comodule is a $B$-bimodule $\Gamma$, together with a $B$-bimodule map ${}_{L}\delta: \Gamma\to \cL\times_{B}\Gamma\subseteq\cL\di \Gamma$, written ${}_{L}\delta(p)=p\mo\ot p\z$ (${}_{L}\delta$ is a $B$-bimodule map in the sense that
${}_{L}\delta(bpb')=bp\mo b'\ot p\z$), such that
\begin{align*}
    &\alpha\circ (\Delta\times_{B}\id_{\Gamma})\circ {}_{L}\delta=\alpha'\circ(\id_{\cL}\times_{B}{}_{L}\delta)\circ {}_{L}\delta:\Gamma\to \cL\times_{B} \cL \times_{B} \Gamma, p\mapsto p\mt\ot p\mo\ot p\z\\
    &\varepsilon(p\mo)p\z=p,
\end{align*}
   for any $p\in \Gamma$. A left $\cL$-comodule $\Gamma$ is called skew regular, if
\[\phi: \cL\ot^{B} P\to \cL\di P, \quad \phi(X\ot p)=p\mo X\ot p\z\]
is invertible. We adopt the shorthand
\begin{align}
    p\ro\ot^{B} p\rz=\phi^{-1}(1\di p).
\end{align}

\end{defi}

\begin{rem}
    Since a left bialgebroid $\cL$ is in particular a coring over $B$, there is already a notion of comodule over $\cL$. By definition, such a comodule over the coring $\cL$ has just one $B$-module structure. As proved in [Section 1.4 of \cite{HHP}], the two notions coincide: A left comodule over $\cL$ in the sense of the definition above is quite obviously a left comodule over the coring $\cL$, but also given a comodule over the coring $\cL$, there is a unique second $B$-module structure which makes it a comodule in the bialgebroid sense.
\end{rem}

\begin{defi}
    Let $\cL$ be a left bialgebroid over $B$, a right $\cL$-comodule is a $\BB$-bimodule $\Gamma$, together with a $\BB$-bimodule map $\delta_{L}: \Gamma\to \Gamma\times_{B}\cL\subseteq\Gamma\di \cL$, written $\delta_{L}(p)=p \z\di p\o$ ($\delta_{L}$ is a $\BB$-bimodule map in the sense that $\delta_{L}(\Bar{b}p\Bar{b'})=p\z\ot \Bar{b}p\o\Bar{b'}$), such that
    \begin{align*}
        &\alpha\circ (\delta_{L}\times_{B}\id_{\cL})\circ \delta_{L}=\alpha'\circ (\id_{\Gamma}\times_{B} \Delta)\circ \delta_{L}:\Gamma\to \Gamma\times_{B} \cL \times_{B} \cL, p\mapsto p\z\ot p\o\ot p\t\\
    &\overline{\varepsilon(p\o)}p\z=p,
    \end{align*}
for any $p\in \Gamma$. A right $\cL$-comodule $\Gamma$ is a regular right $\cL$-comodule, if
\[\psi: P\ot_{\BB} \cL\to P\di \cL, \quad \psi(p\ot_{\BB} X)=p\z \di p\o X\]
is invertible. We adopt the shorthand
\begin{align}
    p\rz\ot_{\BB} p\rmo=\psi^{-1}(P \di 1).
\end{align}
\end{defi}

\begin{rem}
   The definition of a skew regular left comodule is such that $\cL$ is skew regular as a left comodule over itself if and only if it is a anti-left Hopf algebroid. If $B=k$ and $\cL$ is a Hopf algebra with bijective antipode, then  any left $\cL$-comodule $\Gamma$ is skew regular and $\forall p\in \Gamma$, $\phi^{-1}(X\ot p)=S^{-1}(p\mo)X\ot p\z$. In this case, all the right comodules of $\cL$ are regular, with $\psi^{-1}(p\ot X)=p\z\ot S(p\o)X$.
   \end{rem}
In the following proposition, we will see the properties of the inverse of the map $\phi$ generalize the properties of the inverse of the anti-Galois map.

\begin{prop}\label{prop. properties of skew regular comodules}
    Let $\Gamma$ be a skew regular left comodule of a left $B$-bialgebroid $\cL$, then we have
    \begin{align}
        \label{equ. skew regular map 1}
        p\rz\mo p\ro\di p\rz\z=&1\di p,\\
        \label{equ. skew regular map 2}
        p\z\ro p\mo\ot^{B} p\z\rz=&1\ot^{B}p,\\
        \label{equ. skew regular map 4}
        (bpb')\ro\ot^{B} (bpb')\rz=&\Bar{b'}p\ro \Bar{b}\ot^{B} p\rz,\\
        \label{equ. skew regular map 5}
        p\ro b \ot^{B} p\rz=&p\ro\ot^{B} bp\rz
    \end{align}
    Moreover, $\Gamma$ is a right comodule of $\cL$, with coaction $\delta_{L}(p):=p\rz\ot p\ro$,  for any $p\in \Gamma$ and $b, b'\in B$.
\end{prop}
\begin{proof}
   We only show $\Gamma$ is a right $\cL$-comodule with the $\BB$-bimodule structure given  by $\Bar{b}p\Bar{b'}:=b'pb$ for any $b,b'\in B$ and $p\in \Gamma$.  We can check the coaction map is well defined and the image belongs to the Takeuchi product by (\ref{equ. skew regular map 5}). By (\ref{equ. skew regular map 4}) the coaction is $B$-bilinear. First, we can see
   \begin{align*}
      \overline{\varepsilon(p\ro)} p\rz =& p\rz \varepsilon(p\ro)=\varepsilon(p\rz\mo) p\rz\z \varepsilon(p\ro)=\varepsilon(p\rz\mo \overline{\varepsilon(p\ro)}) p\rz\z\\
      =&\varepsilon(p\rz\mo p\ro)p\rz\z=p.
   \end{align*}
    We can also see
    \begin{align}\label{equ. skew regular map 6}
        p\ro\o\ot p\ro\t\ot p\rz=p\rz\ro\ot p\ro \ot p\rz\rz\in \int_{a,b}{}_{a,\Bar{b}}\cL_{a}\ot{}_{b}\cL\ot \Gamma_{a},
    \end{align}
    for any $p\in \Gamma$. Define a bijective map $\phi_{12,3}:\int_{a,b}{}_{a,\Bar{b}}\cL_{a}\ot{}_{b}\cL\ot \Gamma_{a}\to \cL\di{}\cL\di \Gamma$ by $\phi_{12,3}(X\ot Y\ot p)=p\mt X\ot p\mo Y \ot p\z$. Then on the one hand, by applying $\phi_{12,3}$ on the left hand side of (\ref{equ. skew regular map 6}), we get
    \begin{align*}
        \phi_{12,3}(p\ro\o\ot p\ro\t\ot p\rz)=1\ot 1\ot p.
    \end{align*}
    On the other hand, by applying $\phi_{12,3}$ on the right hand side of (\ref{equ. skew regular map 6}), we get
    \begin{align*}
        \phi_{12,3}(p\rz\ro\ot p\ro \ot p\rz\rz)=&p\rz\rz\mt p\rz\ro\ot p\rz\rz\mo p\ro\ot p\rz\rz\z\\
        =&p\rz\rz\mo p\rz\ro\ot p\rz\rz\z\mo p\ro\ot p\rz\rz\z\z\\
        =&1 \ot p\rz\mo p\ro\ot p\rz\z\\
        =&1\ot 1\ot p.
    \end{align*}
    So we have (\ref{equ. skew regular map 6}).
\end{proof}

\begin{lem}\label{lem. anti-left Hopf gives skew regular}
    If $\cL$ is an anti-left $B$-Hopf algebroid, then any left $\cL$-comodule $\Gamma$ is skew regular. More precisely,
    \[p\ro\ot^{B}p\rz=p\mo{}_{[-]}\ot^{B} \varepsilon(p\mo{}_{[+]})p\z.\]
    Moreover, we have
    \begin{align}\label{equ. skew regular map 7}
        p\ro\ot^{B} p\rz\mo\di{} p\rz\z=p\mo{}_{[-]}\ot^{B} p\mo{}_{[+]}\di{} p\z\in \int^{a,d}\int_{b,c}{}_{b}\cL_{a}\ot{}_{a,\Bar{c}}\cL_{b,\Bar{d}}\ot {}_{c}\Gamma_{d},
    \end{align}
    for any $p\in\Gamma$.
\end{lem}
\begin{proof}

Since $X_{[-]}\ot X_{[+]}\in \int^{a}\int_{b}{}_{b}\cL_{a}\ot{}_{a,}\cL_{b}$, and $p\mo\ot p\z\in \int^{a}\int_{b}{}_{\Bar{b}}\cL_{\Bar{a}}\ot {}_{b}\cL_{a}$, we can see $p\mo{}_{[-]}\ot^{B} p\mo{}_{[+]}\di{} p\z\in \int^{a,d}\int_{b,c}{}_{b}\cL_{a}\ot{}_{a,\Bar{c}}\cL_{b,\Bar{d}}\ot {}_{c}\Gamma_{d}$. Let $X\ot Y\ot p \in \int^{a,d}\int_{b,c}{}_{b}\cL_{a}\ot{}_{a,\Bar{c}}\cL_{b,\Bar{d}}\ot {}_{c}\Gamma_{d}$, we can see the map $\phi_{(3.5)}:\int^{a,d}\int_{b,c}{}_{b}\cL_{a}\ot{}_{a,\Bar{c}}\cL_{b,\Bar{d}}\ot {}_{c}\Gamma_{d}\to \int^{a}\int_{b}{}_{b}\cL_{a}\ot {}_{a}\Gamma_{b}$ given by
\begin{align}
    \phi_{(3.5)}(X\ot Y\ot p)=X\ot \varepsilon(Y)p
\end{align}
is well defined. Indeed,
\begin{align*}
    \phi_{(3.5)}(bX\ot Y\ot p)=&bX\ot \varepsilon(Y)p=X\ot \varepsilon(Y)pb=X\ot \varepsilon(Y\Bar{b})p\\
    =&X\ot \varepsilon(Yb)p=\phi_{(3.2)}(X\ot Yb\ot p),
\end{align*}
and
\begin{align*}
\phi_{(3.5)}(X\ot \Bar{b}Y\ot p)=X\ot \varepsilon(\Bar{b}Y)p=X\ot \varepsilon(Y)bp
=\phi_{(3.5)}(X\ot Y\ot bp),
\end{align*}
so $\phi_{(3.5)}$ factors through all the balanced tensor products. Moreover,
\begin{align*}
    Xa\ot \varepsilon(Y)p=&\phi_{(3.5)}(X a\ot Y\ot p)=\phi_{(3.5)}(X\ot aY\ot p)=X\ot \varepsilon(aY)p\\
    =&X\ot a\varepsilon(Y)p,
\end{align*}
for any $X\ot Y\ot p\in \int^{a,d}\int_{b,c}{}_{b}\cL_{a}\ot{}_{a,\Bar{c}}\cL_{b,\Bar{d}}\ot {}_{c}\Gamma_{d}$. To see  $\phi$ is invertible, we have
    on the one hand,
    \begin{align*}
        \phi(p\mo{}_{[-]}\ot^{B} \varepsilon(p\mo{}_{[+]})p\z)=& \varepsilon(p\mt{}_{[+]})p\mo p\mt{}_{[-]}\di{}p\z\\
        =&\varepsilon(p\mo{}_{[+]}\o)p\mo{}_{[+]}\t p\mo{}_{[-]}\di{}p\z\\
        =&p\mo{}_{[+]}p\mo{}_{[-]}\di{}p\z\\
        =&1\di{}p.
    \end{align*}
    On the other hand,
    \begin{align*}
        p\mo{}_{[-]} p\mt\ot^{B} \varepsilon(p\mo{}_{[+]})p\z=1\ot^{B} \varepsilon(p\mo)p\z=1\ot^{B} p.
    \end{align*}
    We can prove (\ref{equ. skew regular map 7}) by comparing the results of
 $\mu\ot \id_{P}$ on both sides of (\ref{equ. skew regular map 7}).

\end{proof}

\begin{prop}
    If $\cL$ is a left $B$-Hopf algebroid and $\Gamma$ is a skew regular left $\cL$-comodule, then we have
    \begin{align}\label{equ. skew regular map 8}
        p\mo\ot p\z\ro \ot p\z\rz=p\ro{}_{-}\ot p\ro{}_{+} \ot p\rz\in \int^{a,b}\int_{c,d}{}_{\Bar{c}}\cL_{\Bar{a}}\ot{}_{\Bar{a},d}\cL_{b, \Bar{c}}\ot{}_{b}\Gamma_{d},
    \end{align}
    for any $p\in \Gamma$.
\end{prop}
\begin{proof}
It is not hard to see $(\lambda\circ \textup{flip}\ot \id):\int^{a,b}\int_{c,d}{}_{\Bar{c}}\cL_{\Bar{a}}\ot{}_{\Bar{a},d}\cL_{b, \Bar{c}}\ot{}_{b}\Gamma_{d}\to \int^{a,b}\int_{c,d}{}_{\Bar{c}, d}\cL_{\Bar{a},b}\ot{}_{c}\cL_{a}\ot{}_{b}\Gamma_{d}$ is a well defined map. Indeed, let $X\ot Y\ot p\in \int^{a,b}\int_{c,d}{}_{\Bar{c}}\cL_{\Bar{a}}\ot{}_{\Bar{a},d}\cL_{b, \Bar{c}}\ot{}_{b}\Gamma_{d}$. We have
\begin{align*}
    (\lambda\circ \textup{flip}\ot \id)(\Bar{b}X\ot Y\ot p)=Y\o\ot Y\t \Bar{b}X\ot p=(\lambda\circ \textup{flip}\ot \id)(X\ot Y\Bar{b} \ot p),
\end{align*}
and
\begin{align*}
     (\lambda\circ \textup{flip}\ot \id)(X\ot bY\ot p)=bY\o\ot Y\t X\ot p=Y\o\ot Y\t X\ot pb= (\lambda\circ \textup{flip}\ot \id)(X\ot Y\ot pb).
\end{align*}
Applying $(\lambda\circ \textup{flip}\ot \id)$ on both sides of (\ref{equ. skew regular map 8}), on the one hand we have,
\begin{align*}
    (\lambda\circ \textup{flip}\ot \id)(p\mo\ot p\z\ro \ot p\z\rz)=&p\z\ro\o\ot p\z\ro\t p\mo\ot p\z\rz\\
    =&p\z\rz\ro\ot p\z\ro p\mo\ot p\z\rz\rz\\
    =&p\ro\ot 1\ot p\rz,
\end{align*}
where the 2nd step uses (\ref{equ. skew regular map 6}). On the other hand,
\begin{align*}
      (\lambda\circ \textup{flip}\ot \id)(p\ro{}_{-}\ot p\ro{}_{+} \ot p\rz)=p\o\ot 1\ot p\rz.
\end{align*}
\end{proof}

Similarly, we can prove

\begin{prop}
    Let $\Gamma$ be a regular right comodule of a left $B$-bialgebroid $\cL$, then we have
    \begin{align}
        \label{equ. regular map 1}
        p\rz\di{} p\rz\o p\rmo=&1\di{}p,\\
        \label{equ. regular map 2}
        p\z\rz\ot_{\BB} p\z\rmo p\o=&p\ot_{\BB}1,\\
        \label{equ. regular map 4}
        (\Bar{b'}p\Bar{b})\rz\ot_{\BB} (\Bar{b'}p\Bar{b})\rmo=&p\rz\ot_{\BB}bp\rmo b',\\
        \label{equ. regular map 5}
        \Bar{b}p\rz\ot_{\BB}p\rmo=&p\rz\ot_{\BB}p\rmo \Bar{b},
    \end{align}
    Moreover, $\Gamma$ is a left comodule of $\cL$, with $B$-bimodule structure  $b'pb:=\Bar{b}p\Bar{b'}$ and  coaction ${}_{L}\delta(p):=p\rmo\ot{}p\rz$. And satisfies
    \begin{align}\label{equ. regular map 6}
        p\rz\ot p\rmo\o\ot{} p\rmo\t=p\rz\rz\ot p\rmo \ot p\rz\rmo,
    \end{align}
    for any $p\in \Gamma$ and any $b, b'\in B$.
\end{prop}

\begin{lem}\label{lem. left Hopf gives regular}
    If $\cL$ is a left $B$-Hopf algebroid, then any right $\cL$-comodule $\Gamma$ is regular. More precisely,
    \[p\rz\ot^{B} p\rmo=\overline{\varepsilon(p\o{}_{+})}p\z \ot^{B}p\o{}_{-} \]
    Moreover, we have
    \begin{align}\label{equ. regular map 7}
        p\rz\z\ot{}p\rz\o\ot p\rmo=p\z\ot{} p\o{}_{+}\ot p\o{}_{-},
    \end{align}
    for any $p\in\Gamma$.
\end{lem}

\begin{prop}
    If $\cL$ is an anti-left $B$-Hopf algebroid and $\Gamma$ is a regular right $\cL$-comodule, then we have
    \begin{align}\label{equ. regular map 8}
        p\z\rz\ot p\z\rmo\ot p\o=p\rz\ot p\rmo{}_{[+]}\ot p\rmo{}_{[-]},
    \end{align}
    for any $p\in \Gamma$.
\end{prop}

\subsection{Left Hopf Galois extensions}

\begin{defi}
Given a left bialgebroid over $B$,    a left $\cL$-comodule algebra $P$ is a $B$-ring and a left $\cL$-comodule, such that the coaction is a $B$-ring map. Let $N={}^{co\cL}P$ be the left invariant subalgebra of $P$, $N\subseteq P$ is called a left $\cL$-Galois extension if the left canonical map ${}^{L}\can: P\ot_{N} P\to \cL\di{}P$ given by
    \[{}^{L}\can(p\ot_{N}q)=p\mo\di{}p\z q,\]
    is bijective. If $N\subseteq P$ is a
left $\cL$-Galois extension, the inverse of ${}^{L}\can$ can be determined by the left translation map:
    \begin{align}
        {}^{L}\tau:={}^{L}\can^{-1}|_{\cL\di{}1}:\cL\to P\ot_{N}P, \qquad X\mapsto \tuno{X}\otimes_{N}\tdue{X}.
    \end{align}
    If the underlying $\cL$-comodule structure of $P$ is skew regular, we call $N\subseteq P$ a skew regular $\cL$-Galois extension. If $P$ is a faithfully flat left $N$-module, we call $N\subseteq P$ a faithfully flat $\cL$-Galois extension.
\end{defi}
 We can see $\can\inv(X\ot p)=\tuno{X}\otimes\tdue{X}p$.
We note that $^L\can$ is a left $B$-module map in the sense that \begin{equation}\label{Lcanlin}
        {^L\can}(p\ot bq)=p_{(-1)}\overline b\ot p_{(0)}q
    \end{equation}
    because the image of the coaction map lies in the Takeuchi product. In addition, the map is obviously right $P$-linear, and also left $B$-linear in a suitable sense, because the coaction is.
As a generalization of \cite{schau2}, we have

\begin{lem}\label{lem. left Hopf Galois induce left Hopf}
    If a left $B$-bialgebroid $\cL$ admits a faithfully flat left $\cL$-Galois extension $N\subseteq P$, then $\cL$ is a left Hopf algebroid. More precisely,
    \[X_{+}\ot X_{-}\ot 1=\tuno{X}\mo\ot \tdue{X}\mo\ot \tuno{X}\z \tdue{X}\z\in \int_{a,b} {}_{\Bar{b}}\cL_{\Bar{a}}\ot{}_{\Bar{a}}\cL\ot {}_{b}P.\]

\end{lem}

\begin{proof}
        It is sufficient to  observe the following diagram  commutes:
     \[
\begin{tikzcd}
  &P\ot_{N} P\ot_{N}P \arrow[d, "P\ot {}^{L}\can"] \arrow[r, "{}^{L}\can\ot P"] & \cL\di{} P \ot_{N}P \arrow[dd, "\cL\ot {}^{L}\can"] &\\
  &P\ot_{N} (\cL \di{} P) \arrow[d, " {}^{L}\can_{1,3}"]\quad&\qquad&\\
   & \cL\ot_{\BB} \cL \di{}P   \arrow[r, "\lambda\ot P"] & \cL\di{} \cL \di{}P, &
\end{tikzcd}
\]
where the map ${}^{L}\can_{1,3}: P\ot_{N} (\cL \di{} P)\to \cL\ot_{\BB}\cL\di{}P$ is given by
\begin{align*}
    {}^{L}\can_{1,3}(p\ot X\ot q)=p\mo\ot X\ot p\z q.
\end{align*}
All maps in the diagram are well-defined as they can be viewed as tensor product of module maps with $P$ or $\cL$. For the map $^ L\can_{1,3}$ in particular this is due to the linearity in \eqref{Lcanlin}, up to an identification. Now by flatness all the maps except for $\lambda\otimes P$ are isomorphisms by assumption, and so by faithful flatness of $P$ we can conclude that $\lambda$ is an isomorphism as well.


\end{proof}

\begin{lem}\label{lem. left skew regular Hopf Galois induce anti-left Hopf}
    If a left $B$-bialgebroid $\cL$ admits a faithfully flat skew regular left $\cL$-Galois extension $N\subseteq P$, then $\cL$ is an anti-left Hopf algebroid. More precisely,
    \[X_{[+]}\ot 1\ot X_{[-]}=\tuno{X}\rz\mo\ot \tuno{X}\rz\z\tdue{X}\ot \tuno{X}\ro \in \int_{a,b} {}_{\Bar{a}}\cL_{b}\ot {}_{a}P\ot {}_{b}P.\]
\end{lem}
\begin{proof}
    It is sufficient to observe the following diagram commute:
     \[
\begin{tikzcd}
  &\cL\ot^{B} P\ot_{N}P \arrow[d, "\cL\ot {}^{L}\can"] \arrow[r, "\phi\ot P"] & \cL\di{} P \ot_{N}P \arrow[d, "\cL\ot {}^{L}\can"] &\\
   & \cL\ot^{B} \cL \di{}P   \arrow[r, "\mu\ot P"] & \cL\di{} \cL \di{}P, &
\end{tikzcd}
\]
it is not hard to see all the maps expect for $\mu\ot P$ are well defined isomorphism, so by the faithful flatness of $P$, $\mu$ is bijective.

\end{proof}

There is a useful proposition for Hopf Galois extensions below:
\begin{prop}
    Let $N\subseteq P$ be a faithfully flat left $\cL$-Galois extension, then we have

\begin{align}\label{equ. translation map 1}
  \tuno{X}\mo \di \tuno{X}\z \ot_{N} \tdue{X} &= X\o\di{}\tuno{X\t}\ot_{N}  \tdue{X\t},\\
\label{equ. translation map 2}
~~ \tdue{X}\mo\ot{}\tuno{X}  \ot \tdue{X}\z &= X_{-}\ot{}\tuno{X_{+}}\ot\tdue{X_{+}}\in \int_{b}{}_{\Bar{b}}\cL\ot P\ot_{N} {}_{b}P,\\
\label{equ. translation map 3}
\tuno{X}\mo\di{}  \tuno{X}\z  \tdue{X} &= X\di{}1,\\
\label{equ. translation map 4}
    \tuno{p\mo}\ot_{N}\tdue{p\mo}p\z&=p\ot_{N}1,\\
\label{equ. translation map 4.5}
    nX\tuno{}\ot_{N}X\tdue{}=&
    X\tuno{}\ot_{N}X\tdue{}n,\\
\label{equ. translation map 5}
\tuno{(aX b )}\ot_{N}\tdue{(aX b)}
&=a\tuno{X}b\ot_{N}\tdue{X},\\
\label{equ. translation map 6}
\tuno{(\Bar{a}X  \Bar{b})}\ot_{N}\tdue{(\Bar{a}X  \Bar{b})}
&=\tuno{X}\ot_{N}b\tdue{X}a,\\
\label{equ. translation map 6.5}
\tuno{X}\tdue{X}
&=\varepsilon(X),\\
\label{equ. translation map 7}
\tuno{(XY)}\ot_{N}\tdue{(XY)}&=\tuno{X}\tuno{Y}\ot_{N}\tdue{Y}\tdue{X},\\
\label{equ. translation map 7.5}
\tuno{X_{+}}\ot_{N} \tuno{X_{-}}\ot_{N} \tdue{X_{-}}\tdue{X_{+}}&=\tuno{X}\ot_{N}\tdue{X}\ot_{N} 1,
\end{align}
for any $X, Y\in \cL$, $p\in P$, $n\in N$ and $a, b\in B$. If $P$ is a skew regular faithfully flat left $\cL$-Galois extension, then we have
\begin{align}
    \label{equ. skew translation map 1}
    p\rz \tuno{p\ro}\ot_{N} \tdue{p\ro}=&1\ot_{N}p,\\
    \label{equ. skew regular map 3}
        (pq)\ro\ot^{B}(pq)\rz=&q\ro p\ro\ot^{B} p\rz q\rz,\\
    \label{equ. skew translation map 2}
    \tuno{X}\ot_{N}\tdue{X}\rz\di{}\tdue{X}\ro=&\tuno{X\o}\ot_{N}\tdue{X\o}\di{}X\t,\\
    \label{equ. skew translation map 3}
    \tuno{X}\rz\ot_{N}\tdue{X}\ot{} \tuno{X}\ro=&\tuno{X_{[+]}}\ot_{N}\tdue{X_{[+]}}\ot{} X_{[-]}\in\int_{b}P_{b}\ot_{N}P\ot {}_{b}\cL,\\
    \label{equ. skew translation map 4}
    \tuno{X_{[+]}} \tuno{X_{[-]}}\ot_{N} \tdue{X_{[-]}}\ot_{N} \tdue{X_{[+]}}=&1\ot_{N}\tuno{X}\ot_{N} \tdue{X}.
\end{align}
\end{prop}

\begin{proof}
We can see (\ref{equ. translation map 3}) and (\ref{equ. translation map 4}) can be given by the definition of translation map. (\ref{equ. translation map 4.5}), (\ref{equ. translation map 5}) and (\ref{equ. translation map 6}) can be shown as the image of the canonical map on both hand sides of the equalities are the same.
    To show (\ref{equ. translation map 2}), we observe that $P\ot_{N}P$ is a left $\cL$-comodule with the coaction given by
    \[{}_{L}\tilde{\delta}(p\ot_{N}q)=q\mo\di{}(p\ot_{N}q\z),\]
   where the underlying $B$-bimodule structure on $P\ot_{N}P$ is given by $b.(p\ot_{N}{}q).b'=p\ot_{N}bqb',$ for any $b, b'\in B$ and $p, q\in P$, it is not hard to see the image of the coaction belongs to the Takeuchi product. We can also see $\cL\di{}P$ is  a $\cL$-comodule with the coaction given by
   \[
   {}_{L}\delta(X\di{}p)=X_{-}p\mo\di{}(X_{+}\di{}p\z),
   \]
where the underlying $B$-bimodule structure of $\cL\di P$ is given by $b.(X\di{}p).b'=X \Bar{b}\di{}pb'$, for any $b, b'\in B$, $X\in\cL$
   and $p\in P$. It is not hard to see ${}_{L}\delta$ is $B$-bimodule map with its image belonging to the Takeuchi product and ${}_{L}\delta$ also factors through all the balanced tensor products. Moreover, by using
   \[X_{+}\ot\one{X_{-}}\ot{}\two{X_{-}}=X_{++}\ot X_{-}\ot{}X_{+-},\] we can see the coaction is coassociative. It is not hard to see ${}^{L}\can$ is a $B$-bimodule map from $P\ot_{N}P$ to $\cL\di{}P$ and moreover, it is a left $\cL$-comodule map. So the translation map is a $\cL$-colinear map as well and this results (\ref{equ. translation map 2}). We can see (\ref{equ. translation map 7.5}) can be proved by (\ref{equ. translation map 2}) and (\ref{equ. translation map 4}). If $P$ is a skew regular left $\cL$-comodule, then we can see $p\rz\tuno{p\ro}\ot_{N}\tdue{p\ro}\in N\ot_{N}P$. Indeed,
   \begin{align*}
       ({}_{L}\delta\ot_{N}\id) (p\rz\tuno{p\ro}\ot_{N}\tdue{p\ro})=&p\rz\mo \tuno{p\ro}\mo\ot p\rz\z \tuno{p\ro}\z\ot \tdue{p\ro}\\
       =&p\rz\mo p\ro\o\ot p\rz\z\tuno{p\ro\t}\ot \tdue{p\ro\t}\\
       =&p\rz\rz\mo p\rz\ro\ot p\rz\rz\z\tuno{p\ro}\ot \tdue{p\ro}\\
       =&1\ot p\rz\tuno{p\ro}\ot \tdue{p\ro},
   \end{align*}
   where the 2nd step uses (\ref{equ. translation map 1}), the 3rd step uses (\ref{equ. skew regular map 6}). So we can show (\ref{equ. skew translation map 1}), since $p\rz\tuno{p\ro}\ot_{N} \tdue{p\ro}=1\ot_{N} p\rz\tuno{p\ro} \tdue{p\ro}=1\ot_{N}p$. We can show (\ref{equ. skew translation map 2}) by comparing the results of the map $\id\ot \phi$ on both sides of (\ref{equ. skew translation map 2}). By Lemma \ref{lem. left skew regular Hopf Galois induce anti-left Hopf}, $\cL$ is an anti-left Hopf algebroid, so we can show (\ref{equ. skew translation map 3}) by comparing the results of the map ${}^{L}\can\ot{} \id$ on both sides of (\ref{equ. skew translation map 3}). By applying $\id\ot {}^{L}\can$ on the left hand side of (\ref{equ. skew translation map 4}), we get
   \begin{align*}
       \tuno{X_{[+]}} &\tuno{X_{[-]}}\ot \tdue{X_{[-]}}\mo\ot \tdue{X_{[-]}}\z\tdue{X_{[+]}}\\
       =&\tuno{X_{[+]}} \tuno{X_{[-]+}}\ot X_{[-]-}\ot \tdue{X_{[-]+}}\tdue{X_{[+]}}\\
       =&\tuno{X\t{}_{[+]}} \tuno{X\t{}_{[-]}}\ot X\o\ot \tdue{X\t{}_{[-]}}\tdue{X\t{}_{[+]}}\\
       =&1\ot X\ot 1,
   \end{align*}
   where the 1st step uses (\ref{equ. translation map 2}), the 2nd step uses Proposition \ref{prop. anti left and left Galois maps}. Since we have the same result by applying on the right hand side of (\ref{equ. skew translation map 4}), so we have (\ref{equ. skew translation map 4}).
\end{proof}

Recall that for a left comodule algebra $P$ of a left $B$-bialgebroid $\cL$, a left-left (resp. left-right) relative Hopf module $M\in {}^{\cL}_{P}\M$ (resp. ${}^{\cL}\M_{P}$) is a left (resp. right) $P$-module in the category of left $\cL$-comodules. We also have generalized the theorem given by Schneider \cite{schneider}.

\begin{thm}\label{thm. fundamental theorem for left Hopf Galois extensions}
Let $\cL$ be a left $B$-bialgebroid and $N\subseteq P$ be a faithfully flat left $\cL$-Galois extension, then
\begin{align*}    {}^{\cL}_{P}\M\to {}_{N}\M,\qquad& M\mapsto {}^{co\cL}M\\
    {}_{N}\M\to {}^{\cL}_{P}\M,\qquad& \Lambda\mapsto P\ot_{N}\Lambda
\end{align*}
   are quasi-inverse category equivalences. Moreover, if $P$ is skew regular left $\cL$-comodule, then
\begin{align*}    {}^{\cL}\M_{P}\to \M_{N},\qquad& M\mapsto {}^{co\cL}M\\
    \M_{N}\to {}^{\cL}\M_{P},\qquad& \Lambda\mapsto \Lambda\ot_{N}P
\end{align*}
   are quasi-inverse category equivalences.
\end{thm}
\begin{proof}
 The $\cL$-comodule and $P$-module structure on $P\ot_{N}\Lambda$ is given by the structure on $P$.     The isomorphism $\Lambda\simeq {}^{co\cL}(P\ot_{N}\Lambda)$ for any $\Lambda\in {}_{N}\M$ can be given by $\eta\mapsto 1\ot \eta$. The isomorphism $M\simeq P\ot_{N}{}^{co\cL}M$ for any $M\in {}^{\cL}_{P}\M$ can be given by $m\mapsto \tuno{m\mo}\ot_{N}\tdue{m\mo} m\z$ with inverse $p\ot_{N}\eta\mapsto p \eta$.

    In the case that $P$ is skew regular, the $\cL$-comodule and $P$-module structure on $\Lambda\ot_{N}P$ is given by the structure on $P$. The isomorphism $\Lambda\simeq {}^{co\cL}(\Lambda\ot_{N}p)$ for any $\Lambda\in \M{}_{N}$ can be given by $\eta\mapsto \eta\ot 1$. The isomorphism $M\simeq {}^{co\cL}M\ot_{N}P$ for any $M\in {}^{\cL}\M_{P}$ can be given by $m\mapsto m\rz \tuno{m\ro}\ot_{N}\tdue{m\ro}$ with inverse $\eta\ot_{N}p\mapsto \eta p$.
\end{proof}

\subsection{Anti-right Hopf Galois extensions}

We also have right comodules of left bialgebroids, with slightly different properties.

\begin{defi}
Given a left $B$-bialgebroid $\cL$,  a right $\cL$-comodule algebra $P$ is a $\BB$-ring and a right $\cL$-comodule, such that the coaction is a $\BB$-ring map. Let $M:=P{}^{co\cL}$ be the right invariant subalgebra of $P$, $M\subseteq P$ is called a anti-right $\cL$-Galois extension if the anti-right canonical map $\hat{\can}^{L}: P\ot_{M} P\to P\di{}\cL$ given by
    \[\hat{\can}^{L}(p\ot_{M}q)=p\z q\di{}p\o ,\]
    is bijective. If $M\subseteq P$ is a
anti-right $\cL$-Galois extension, the inverse of $\hat{\can}^{L}$ can be determined by the anti-right translation map:
    \begin{align}
        \hat{\tau}^{L}:=(\hat{\can}^{L})^{-1}|_{1\di{}\cL}:\cL\to P\ot_{M}P, \qquad X\mapsto \yi{X}\otimes_{M}\er{X}.
    \end{align}
     If the underlying $\cL$-comodule structure of $P$ is regular, we call $M\subseteq P$ a regular anti-right $\cL$-Galois extension. If $P$ is a faithfully flat right $M$-module, we call $M\subseteq P$ a faithfully flat $\cL$-Galois extension.
\end{defi}

\begin{lem}\label{lem. anti Galois induce anti-left Hopf}
    If a left bialgebroid $\cL$ admits an anti-right $\cL$-Galois extension $M\subseteq P$, then $\cL$ is an anti-left Hopf algebroid. More precisely,
    \[1\ot X_{[+]}\ot X_{[-]}=\yi{X}\z \er{X}\z\ot \yi{X}\o\ot \er{X}\o\in P\di \cL\ot_{B}\cL .\]
\end{lem}

\begin{lem}\label{lem. regular anti Galois induce left Hopf}
    If a left bialgebroid $\cL$ admits a regular anti-right $\cL$-Galois extension $M\subseteq P$, then $\cL$ is a left Hopf algebroid. More precisely,
    \[X_{+}\ot 1\ot X_{-}=\yi{X}\rz\o\ot \yi{X}\rz\z\er{X}\ot \yi{X}\rmo\in \int_{a,b}{}_{a}\cL{}_{\Bar{b}}\ot {}_{\Bar{a}}\cL\ot {}_{\Bar{b}}\cL.\]
\end{lem}

There is a useful proposition for anti-right Hopf Galois extensions below:
\begin{prop}\label{prop. left Hopf Galois extension}
    Let $M\subseteq P$ be a faithfully flat anti-right $\cL$-Galois extension, then we have

\begin{align}\label{equ. anti translation map 1}
  \yi{X}\z \ot_{M} \er{X} \ot{} \yi{X}\o &= \yi{X\o}\ot_{M}\er{X\o}\ot{}X\t\in\int_{b}{}_{\Bar{b}}P\ot_{M}P\ot {}_{b}\cL,\\
\label{equ. anti translation map 2}
~~ \yi{X}\ot_{M}\er{X}\z\di{}\er{X}\o &= \yi{X_{[+]}}\ot_{M}\er{X_{[+]}}\di{}X_{[-]},\\
\label{equ. anti translation map 3}
\yi{X}\z \er{X}\di{}\yi{X}\o &= 1\di{}X,\\
\label{equ. anti translation map 4}
    \yi{p\o}\ot_{M}\er{p\o}p\z&=p\ot_{M}1,\\
 \label{equ. anti translation map 4.5}
    mX\yi{}\ot_{M}X\er{}=&
    X\yi{}\ot_{M}X\er{}m,\\
\label{equ. anti translation map 5}
\yi{(aX b )}\ot_{M}\er{(aXb )}
&=\yi{X}\ot_{M}\Bar{b}\er{X}\Bar{a},\\
\label{equ. anti translation map 6}
\yi{(\Bar{a}X  \Bar{b})}\ot_{M}\er{(\Bar{a}X  \Bar{b})}
&=\Bar{a}\yi{X}\Bar{b}\ot_{M}\er{X},\\
\label{equ. anti translation map 6.5}
\yi{X}\er{X}
&=\overline{\varepsilon(X)},\\
\label{equ. anti translation map 7}
\yi{(XY)}\ot_{M}\er{(XY)}&=\yi{X}\yi{Y}\ot_{M}\er{Y}\er{X},\\
\label{equ. anti translation map 7.5}
\yi{X_{[+]}}\ot_{M} \yi{X_{[-]}}\ot_{M} \er{X_{[-]}}\er{X_{[+]}}&=\yi{X}\ot_{M}\er{X}\ot_{M} 1,
\end{align}
for any $X, Y\in \cL$, $p\in P$, $m\in M$ and $a, b\in B$. If $P$ is a regular right $\cL$-comodule, then we have
\begin{align}
    \label{equ. regular anti translation map 1}
    p\rz \yi{p\rmo}\ot_{M} \er{p\rmo}=&1\ot_{M}p,\\
    \label{equ. regular map 3}
        (pq)\rz\ot_{\BB}(pq)\rmo=& p\rz q\rz\ot_{\BB} q\ro p\ro\\
    \label{equ. regular anti translation map 2}
    \yi{X}\ot_{M}\er{X}\rz\di \er{X}\rmo=&\yi{X\t}\ot_{M}\er{X\t}\di X\o,\\
    \label{equ. regular anti translation map 3}
    \yi{X}\rz\ot_{M}\er{X}\ot{} \yi{X}\rmo=&\yi{X_{+}}\ot_{M}\er{X_{+}}\ot{} X_{-}\in \int_{b}P_{\Bar{b}}\ot_{M}P\ot {}_{\Bar{b}}\cL,\\
    \label{equ. regular anti translation map 4}
    \yi{X_{+}} \yi{X_{-}}\ot \er{X_{-}}\ot \er{X_{+}}=&1\ot\yi{X}\ot \er{X}.
\end{align}
\end{prop}

For anti-right Hopf Galois extensions, we also have generalized the theorem in \cite{schneider}.
\begin{thm}\label{thm. fundamental theorem for anti-right Hopf Galois extensions}
Let $\cL$ be a left $B$-bialgebroid and $N\subseteq P$ be a faithfully flat anti-right $\cL$-Galois extension, then
\begin{align*}    {}_{P}\M^{\cL}\to {}_{N}\M,\qquad& M\mapsto M^{co\cL}\\
    {}_{N}\M\to {}_{P}\M^{\cL},\qquad& \Lambda\mapsto P\ot_{N}\Lambda
\end{align*}
   are quasi-inverse category equivalences. Moreover, if $P$ is regular right $\cL$-comodule, then
\begin{align*}    \M^{\cL}_{P}\to \M_{N},\qquad& M\mapsto M^{co\cL}\\
    \M_{N}\to \M^{\cL}_{P},\qquad& \Lambda\mapsto \Lambda\ot_{N}P
\end{align*}
   are quasi-inverse category equivalences.
\end{thm}
\begin{proof}
    The isomorphism $\Lambda\simeq (P\ot_{N}\Lambda)^{co\cL}$ for any $\Lambda\in \M_{N}$ can be given by $\eta\mapsto 1\ot \eta$. The isomorphism $M\simeq P\ot_{N}M^{co\cL}$ for any $M\in {}^{\cL}\M_{P}$ can be given by $m\mapsto \yi{m\o}\ot_{N}\er{m\o}m\z$ with inverse $p\ot_{N}\eta\mapsto p \eta$.

    In the case that $P$ is regular,  The isomorphism $\Lambda\simeq (\Lambda\ot_{N}P)^{co\cL}$ for any $\Lambda\in \M{}_{N}$ can be given by $\eta\mapsto \eta\ot 1$. The isomorphism $M\simeq M^{co\cL}\ot_{N}P$ for any $M\in {}^{\cL}\M_{P}$ can be given by $m\mapsto m\rz \yi{m\rmo}\ot_{N}\er{m\rmo}$ with inverse $\eta\ot_{N}p\mapsto \eta p$.
\end{proof}

\begin{lem}\label{lem. P opposite has Galois structure}
    Let $\cL$ be a left $B$-bialgebroid and $N\subseteq P$ be a skew regular left $\cL$-Galois extension, then $N^{op}\subseteq P^{op}$ is a regular anti-right $\cL$-Galois extension. More precisely, the right coaction is given by $\delta_{\cL}:p\mapsto p\rz\di{}p\ro$, and the anti-right translation map is
    \[\hat{\tau}^{\cL}:X\mapsto X\tdue{}\ot_{N^{op}}X\tuno{},\]
    for any $p\in P$ and $X\in \cL$.

    Similarly, if $M\subseteq P$ is a regular anti-right $\cL$-Galois extension, then $M^{op}\subseteq P^{op}$ is a skew regular left Hopf Galois extension of $\cL$. More precisely, the left coaction is given by ${}_{\cL}\delta: p\to p\rmo\di{}p\rz$, and the left translation map is
    \[{}^{\cL}\tau: X\to X\er{}\ot_{M^{op}}X\yi,\]
    for any $p\in P$ and $X\in \cL$.
\end{lem}

\begin{proof}
    We only show the first half of the lemma. By Proposition \ref{prop. properties of skew regular comodules}, the coaction is a $\BB$-ring map. Moreover, we can see the right coinvariant subalgebra is $N^{op}$. More precisely, if $n\in N^{op}$, then
    \begin{align*}
        n\rz\di{}n\ro=n\z\rz\di{} n\z\ro n\mo=n\di{} 1.
    \end{align*}
    Conversely, if $m\in (P^{op})^{co\cL}$, then
    \begin{align*}
        m\mo\di{}m\z=m\rz\mo m\ro\di{}m\rz\z=1\di{}m.
    \end{align*}
    By (\ref{equ. skew translation map 1}) and (\ref{equ. skew translation map 2}), we can check $\hat{\tau}^{\cL}$ is the anti-right translation map. The regular right comodule structure is given by the original left $\cL$-coaction.
\end{proof}

\section{Drinfeld twist of Hopf algebroids}\label{sec. Drinfeld twist of Hopf algebroids}
In this section, we will study the Drinfeld twist theory of Hopf Galois extensions.

\begin{defi}\label{def. left handed 2 cocycle} cf. \cite{HM22, HM23,Boehm}
Let $\cL$ be a left $B$-bialgebroid. An \textup{invertible normalised 2-cocycle} on
$\cL$ is a convolution invertible element $\Gamma\in {}_{B^{e}}\Hom(\cL\otimes_{B^e}\cL, B)$ such that
\begin{align*}(i)\quad &\Gamma(X, \Gamma(\one{Y}, \one{Z})\two{Y}\two{Z})=\Gamma(\Gamma(\one{X}, \one{Y})\two{X}\two{Y}, Z),\\
(ii)\quad
 &\Gamma(1_{\cL}, X)=\varepsilon(X)=\Gamma(X, 1_{\cL}),\\
(iii)\quad &\Gamma(X, Yb)=\Gamma(X, Y\Bar{b})\end{align*}
for all $X, Y, Z\in \cL$. The collection of such 2-cocycles of $\cL$ over $B$ will be denoted $Z^{2}(\cL, B)$.
\end{defi}
Here (iii) implies the minimal condition in \cite{HM22} that $\Gamma^{-1}$ should be a right-handed 2-cocycle in the sense
\begin{align}\label{equ. right handed 2-cocycle}
    &\Gamma^{-1}(X, \overline{\Gamma^{-1}(\two{Y}, \two{Z})}\one{Y}\one{Z})=\Gamma^{-1}(\overline{\Gamma^{-1}(\two{X}, \two{Y})}\one{X}\one{Y}, Z),\\ &\Gamma^{-1}(1_{\cL}, X)=\varepsilon(X)=\Gamma^{-1}(X, 1_{\cL}),\\
    &\Gamma^{-1}(X, Yb)=\Gamma^{-1}(X, Y\Bar{b})
\end{align}
and also implies that $\Gamma^{-1}$ obeys (ii).
By \cite{HM22}, we have
\begin{lem}\label{lemma. 2 cocycle and its inverse}
Let $\cL$ be a left bialgebroid and $\Gamma\in {}_{B^{e}}\Hom(\cL\otimes_{B^e}\cL, B)$ be an invertible normalised  2-cocycle with inverse $\Gamma^{-1}$. Then
\begin{align}
    \Gamma(\one{X}, \one{Y}\one{Z})\Gamma^{-1}(\two{X}\two{Y}, \two{Z})&=\Gamma(X\Gamma^{-1}(\one{Y}, Z), \two{Y}),\label{equ. 2 cocycle and its inverse1}\\
    \Gamma(\one{X}\one{Y}, \one{Z})\Gamma^{-1}(\two{X}, \two{Y} \two{Z})&=\Gamma^{-1}(X, \overline{\Gamma(\two{Y}, Z)}\one{Y}).\label{equ. 2 cocycle and its inverse2}
\end{align}
\end{lem}

Let $\cL$ be a left $B$-bialgebroid and $\Gamma\in {}_{B^{e}}\Hom(\cL\otimes_{B^e}\cL, B)$ be an invertible normalised 2-cocycle, we can define a new left $B$-bialgebroid $\cL^\Gamma$ with twisted product
\begin{align}\label{twistprod}
    X\cdot_{\Gamma} Y:=\Gamma(\one{X}, \one{Y})\overline{\Gamma^{-1}(\three{X}, \three{Y})}\two{X}\two{Y},
\end{align}
 and the original coproduct, counit, source, target maps.

\begin{lem}
    Let $\cL$ be a left bialgebroid over $B$, and $\Gamma$ be an invertible normalized 2-cocycle on $\cL$. Then the category of left $\cL$-comodule ${}^{\cL}\mathcal{M}$ is equivalent to the category of left $\cL^{\Gamma}$-comodule ${}^{\cL^{\Gamma}}\mathcal{M}$ as monoidal category.
\end{lem}

\begin{proof}
    Given $V$ and $W$ as left $\cL$-comodule, then $V\ot_{B}{}W$ is also left $\cL$-comodule with diagonal coaction and $B$-bimodule structure given by $b(v\ot w)b'=bv\ot wb'$ for any $b,b'\in B$ and $v\ot w\in V\ot_{B}{}W$. The monoidal functor is the identity on the underlying $B$-bimodule, with the corresponding natural transformation given by
    \[\xi_{V, W}:v\ot_{B}{}w\mapsto \Gamma(v\mo, w\mo)v\z\ot_{B}{}w\z,\]
    for any $v\ot w\in V\ot_{B}{}W$. The coherent condition can be given by the 2-cocycle condition of $\Gamma$.
\end{proof}
If $P$ is a left-$\cL$-comodule algebra, the underlying $\cL$-comodule, with the left side twisted product
\begin{align}\label{equ. left side twisted product}
   p\cdot_{\Gamma}q=\Gamma(p\mo, q\mo)p\z q\z,
\end{align}
is a left $\cL^{\Gamma}$-comodule algebra. We denote the new algebra by ${}_{\Gamma}P$.

\begin{thm}\label{thm. twi have one sted left Hopf Galois extensions}
    Let $N\subseteq P$ be a faithfully flat left $\cL$-Galois extension. Then $N\subseteq {}_{\Gamma}P$ is a left $\cL^{\Gamma}$-Galois extension, with the twisted left translation map given by
    \begin{align}
        {}^{\cL^{\Gamma}}\tau(X)=\tuno{X\o{}_{+}}\ot_{N}\Gamma^{-1}(X\o{}_{-}, X\t)\tdue{X\o{}_{+}},
    \end{align}
    for any $X\in \cL$.
\end{thm}

\begin{proof}
For any $X\in \cL$,

\begin{align*}
   X\o{}_{+}\ot X\o{}_{-}\ot X\t \in \int^{a,b}\int_{c,d}{}_{\Bar{a}}\cL_{\Bar{c}}\ot{}_{b,\Bar{c}}\cL_{\Bar{a}, d}\ot {}_{d}\cL_{b}.
\end{align*}
We can see the map
\begin{align*}
    \phi_{(4.1)}:\int^{a,b}\int_{c,d}{}_{\Bar{a}}\cL_{\Bar{c}}\ot{}_{b,\Bar{c}}\cL_{\Bar{a}, d}\ot {}_{d}\cL_{b}\to \cL
\end{align*}
given by
\begin{align*}
    \phi_{(4.1)}(X\ot Y\ot Z)=X \overline{(\Gamma^{-1}(Y, Z))}
\end{align*}
is well defined. Indeed,
\begin{align*}
    \phi_{(4.1)}(X\ot Y\ot d Z)=&X \overline{(\Gamma^{-1}(Y, dZ))}=X \overline{(\Gamma^{-1}(Yd, Z))}\\
    =&\phi_{(4.1)}(X\ot Y d\ot  Z).
\end{align*}
Also we have,
\begin{align*}
    \phi_{(4.1)}(X \Bar{c}\ot Y\ot  Z)=&X\Bar{c} \overline{(\Gamma^{-1}(Y, Z))}=X \overline{(\Gamma^{-1}(Y, Z)c)}=X \overline{(\Gamma^{-1}(\Bar{c}Y, Z))}\\
    =&\phi_{(4.1)}(X\ot \Bar{c}Y\ot  Z).
\end{align*}
By (\ref{equ. translation map 6}), the twisted translation map ${}^{L^{\Gamma}}\tau(X)={}^{L}\tau\circ \phi_{(4.1)}(X\o{}_{+}\ot X\o{}_{-}\ot X\t)$, so it is well defined as well. Now, let's check it determines the inverse of the left canonical map. On the one hand, we have
    \begin{align*}
        {}^{L^{\Gamma}}\tau\circ {}^{L^{\Gamma}}\can(p\ot_{N}1)=&\tuno{p\mth{}_{+}}\ot_{N}\Gamma^{-1}(p\mth{}_{-}, p\mt)\Gamma(\tdue{p\mth{}_{+}}\mo, p\mo)\tdue{p\mth{}_{+}}\z p\z\\
        =&\tuno{p\mth{}_{++}}\ot_{N}\Gamma^{-1}(p\mth{}_{-}, p\mt)\Gamma(p\mth{}_{+-}, p\mo)\tdue{p\mth{}_{++}} p\z\\
        =&\tuno{p\mth{}_{+}}\ot_{N}\Gamma^{-1}(p\mth{}_{-}\o, p\mt)\Gamma(p\mth{}_{-}\t, p\mo)\tdue{p\mth{}_{+}} p\z\\
        =&\tuno{p\mt{}_{+}}\ot_{N}\varepsilon(p\mt{}_{-} p\mo)\tdue{p\mt{}_{+}} p\z\\
        =&\tuno{p\mt{}_{+}}\ot_{N}\varepsilon(p\mt{}_{-} \overline{\varepsilon(p\mo)})\tdue{p\mt{}_{+}} p\z\\
        =&\tuno{p\mt{}_{+}}\ot_{N}\varepsilon(p\mt{}_{-} )\tdue{p\mt{}_{+}} \varepsilon(p\mo) p\z\\
                =&\tuno{p\mo}\ot_{N}\varepsilon(\tdue{p\mo}\mo )\tdue{p\mo}\z  p\z\\
        =&\tuno{p\mo}\ot_{N}\tdue{p\mo}p\z\\
        =&p\ot_{N}1,
    \end{align*}
  where the 2nd and 7th steps use (\ref{equ. translation map 2}).  On the other hand,
    \begin{align*}
         {}^{L^{\Gamma}}&\can\circ {}^{L^{\Gamma}}\tau (X)\\
         =&X\o{}_{+}\un\mt\di{}\Gamma(X\o{}_{+}\un\mo, \Gamma^{-1}(X\o{}_{-}, X\t)\tdue{X\o{}_{+}}\mo) X\o{}_{+}\un\z \tdue{X\o{}_{+}}\z\\
         =&X\o{}_{++}\un\mt\di{}\Gamma(X\o{}_{++}\un\mo, \Gamma^{-1}(X\o{}_{-}, X\t) X\o{}_{+-}) X\o{}_{++}\un\z X\o{}_{++}\du\\
         =&X\o{}_{+}\un\mt\di{}\Gamma(X\o{}_{+}\un\mo, \Gamma^{-1}(X\o{}_{-}\o, X\t) X\o{}_{-}\t) X\o{}_{+}\un\z X\o{}_{+}\du\\
         =&X\o{}_{+}\o\di{}\Gamma(X\o{}_{+}\t,  \Gamma^{-1}(X\o{}_{-}\o, X\t)X\o{}_{-}\t)\\
         =&X\o{}_{+}\o\di{}\Gamma(X\o{}_{+}\t, X\o{}_{-}\o X\t)\Gamma^{-1}(X\o{}_{+}\th X\o{}_{-}\t, X\th)\\
         =&X\o{}_{++}\o\di{}\Gamma(X\o{}_{++}\t, X\o{}_{-} X\t)\Gamma^{-1}(X\o{}_{++}\th X\o{}_{+-}, X\th)\\
         =&X\o{}_{+}\o\di{}\Gamma^{-1}(X\o{}_{+}\t X\o{}_{-}, X\t)\\
         =&X\di{}1,
    \end{align*}
    where the 2nd step uses (\ref{equ. translation map 2}), the 5th step uses Lemma \ref{thm. 2 cocycle twist1}.
\end{proof}

Similarly, If $P$ is a right comodule algebra of $\cL$, then the underlying right $\cL$-comodule,
with the right side twisted product
\begin{align}
    p{}_{\Gamma}\cdot q=\overline{\Gamma^{-1}(p\o, q\o)}p\z q\z,
\end{align}
is also a right comodule algebra. Moreover, we have
\begin{thm}\label{thm. twi have one sted anti-right Hopf Galois extensions}
    Let $M\subseteq P$ be a faithfully flat anti-right $\cL$-Galois extension. Then $M\subseteq P_{\Gamma}$ is an anti-right $\cL^{\Gamma}$-Galois extension, with the twisted anti-right translation map given by
    \begin{align}
        \hat{\tau}^{\cL^{\Gamma}}(X)=\yi{X\t{}_{[+]}}\ot_{M}\overline{\Gamma(X\t{}_{[-]}, X\o)}\er{X\t{}_{[+]}},
    \end{align}
    for any $X\in \cL$.
\end{thm}

As a result of Lemma \ref{lem. left Hopf Galois induce left Hopf} and Lemma \ref{lem. anti Galois induce anti-left Hopf}, we can show the Theorem already given in \cite{HM22}

\begin{thm}\label{thm. 2 cocycle twist1}
Let $\cL$ be a (anti-)left $B$-Hopf algebroid and $\Gamma\in {}_{B^{e}}\Hom(\cL\otimes_{B^e}\cL, B)$ be an invertible normalised 2-cocycle. Then the left bialgebroid $\cL^\Gamma$ is a (anti-)left Hopf algebroid.
\end{thm}

\begin{cor}
     Let $N\subseteq P$ be a skew regular faithfully flat left $\cL$- Galois extension, then $N\subseteq {}_{\Gamma}P$ is a skew regular left $\cL^{\Gamma}$-Galois extension.

     Similarly, let $M\subseteq P$ be a regular faithfully flat anti-right $\cL$-Galois extension, then $M\subseteq P_{\Gamma}$ is a regular anti-right $\cL^{\Gamma}$-Galois extension.
\end{cor}

\begin{proof}

 By Lemma \ref{lem. left skew regular Hopf Galois induce anti-left Hopf}, we know $\cL$ is an anti left Hopf algebroid. By  Theorem \ref{thm. 2 cocycle twist1}, we know $\cL^{\Gamma}$ is also an anti left Hopf algebroid, so the first half of the corrolary is the result of  Lemma \ref{lem. anti-left Hopf gives skew regular} and Thm \ref{thm. twi have one sted left Hopf Galois extensions}. The second half of the theorem is similar.
\end{proof}

\section{Twisted Drinfeld Doubles}
 By \cite{schau1}, given two left bialgebroids, one can define a skew pairing.

\begin{defi}\label{def. dual pairing for left bialgeboids}
    Let $\cL$ and $\Pi$ be two left bialgebroids over $B$, a skew pairing between $\cL$ and $\Pi$ is a linear map $\tau[\bullet |\bullet] :\cL\ot\Pi\to B$ such that:
    \begin{itemize}
        \item $\tau[a\Bar{b}X c\Bar{d}| \alpha]f=a\tau[X|c\Bar{f}\alpha d\Bar{b}]$,
        \item  $\tau[X|\alpha\beta]=\tau[X\o|\alpha \overline{\tau[X\t|\beta]}]=\tau[\overline{\tau[X\t|\beta]}X\o|\alpha ]$,
        \item $\tau[XY|\alpha]=\tau[X \tau[Y|\alpha\o]|\alpha\t]=\tau[X |\tau[Y|\alpha\o]\alpha\t]$,
        \item $\tau[X|1]=\varepsilon(X)$,
        \item $\tau[1|\alpha]=\varepsilon(\alpha)$,
    \end{itemize}
    for all $a, b, c, d, f\in B$, $\alpha, \beta\in \Pi$ and $X, Y\in \cL$.
\end{defi}
It is given in \cite{schau1} that, given  a skew pairing $\tau[\bullet|\bullet]$ between two left Hopf algebroids $\Pi$ and $\cL$, we can construct a left bialgebroid $\Pi\dcrosstt\cL$ on the vector space $\Pi\ot_{B^{e}}\cL$ with
 componentwise  coproduct and the product
\[(\alpha\dcrosstt X)\cdot_{\tau\tau}(\beta\dcrosstt Y)=\alpha \tau[X\o|\beta\o]\beta\t{}_{+}\dcrosstt  \overline{\tau[X\th|\beta\t{}_{-}]}X\t Y,\]
and the counit
\[\varepsilon(\alpha\dcrosstt X)=\varepsilon(\alpha \varepsilon(X)).\]

\begin{lem}
Let $\tau[\bullet|\bullet]$ be a skew pairing between two  Hopf algebroid $\Pi$ and $\cL$, then  $\Pi\dcrosstt\cL$ is a Hopf algebroid, with
\begin{align}\label{equ. left Galois map for skew pairing}
    (\alpha\dcrosstt X)_{+}\ot_{\BB}(\alpha\dcrosstt X)_{-}=\alpha_{+}\dcrosstt X_{+}\ot_{\BB}(1\dcrosstt X_{-})(\alpha_{-}\dcrosstt 1),
\end{align}
and
\begin{align}\label{equ. anti-left Galois map for skew pairing}
    (\alpha\dcrosstt X)_{[-]}\ot^{B}(\alpha\dcrosstt X)_{[+]}=(1\dcrosstt X_{[-]})(\alpha_{[-]}\dcrosstt{}1)\ot^{B}\alpha_{[+]}\dcrosstt{} X_{[+]},
\end{align}
for any $\alpha\dcrosstt{}X\in {}^{\tau}D(\cL)^{\tau}$.

\end{lem}

\begin{proof}

   It is not hard to see (\ref{equ. left Galois map for skew pairing}) factors through all the balanced tensor products. We can see
\begin{align*}
    (\alpha\dcrosstt{}&X)_{+}\o\di{}(\alpha\dcrosstt{}X)_{+}\t (\alpha\dcrosstt{}X)_{-}\\
    =&(\alpha_{+}\o\dcrosstt{}X_{+}\o)\di{}(\alpha_{+}\t\dcrosstt{}X_{+}\t)(1\dcrosstt{}X_{-})(\alpha_{-}\dcrosstt{}1)\\
    =&(\alpha_{+}\o\dcrosstt{}X_{+}\o)\di{}(\alpha_{+}\t\dcrosstt{}X_{+}\t X_{-})(\alpha_{-}\dcrosstt{}1)\\
    =&(\alpha_{+}\o\dcrosstt{}X)\di{}(\alpha_{+}\t\dcrosstt{}1)(\alpha_{-}\dcrosstt{}1)\\
     =&(\alpha_{+}\o\dcrosstt{}X)\di{}(\alpha_{+}\t\alpha_{-}\dcrosstt{}1)\\
     =&(\alpha\dcrosstt{}X)\di{}(1\dcrosstt{}1).\\
\end{align*}
   Similarly, we have
   \begin{align*}
       (\alpha\dcrosstt{}X)\o{}_{+}\ot_{\BB}(\alpha\dcrosstt{}X)\o{}_{-}(\alpha\dcrosstt{}X)\t=(\alpha\dcrosstt{}X)\ot_{\BB}(1\dcrosstt{}1).
   \end{align*}
   The anti-left Hopf algebroid structure can be similarly proved.
\end{proof}

In fact, by repeating the calculations in \cite{schau1}, one can show:
\begin{lem}\label{gencomult}
  Let $\Pi,\cL$ be two left Hopf algebroids, and $\kappa,\tau,\omega,\nu$ four skew pairings. Then
  \begin{enumerate}
      \item $\Pi\dcrosstk\cL:=\Pi\otimes_{B^e}\cL$ equipped with multiplication
      \begin{align}\label{equ. twist on two side products}
         (\alpha\dcrosstk X)\cdot_{\tau\kappa}(\beta\dcrosstk Y)=\alpha \tau[X\o|\beta\o]\beta\t{}_{+}\dcrosstt \overline{ \kappa[X\th|\beta\t{}_{-}]}X\t Y
      \end{align}
     is a $B^e$-ring.
      \item The map
      \begin{align*}
          \Delta\colon \Pi\otimes_{B^e}\cL &\to (\Pi\otimes_{B^e}\cL)\times_B (\Pi\otimes_{B^e}\cL)\\
          \alpha\ot X&\mapsto \alpha\o\ot X\o\ot \alpha\t\ot X\t
      \end{align*}
      making $\Pi\otimes_{B^e}\cL$ a $B$-coring with counit $\varepsilon(\alpha\ot X)=\varepsilon(\alpha \varepsilon(X))$.
      \item
      \begin{align*}
          \Delta=\Delta_{\tau\omega\kappa}\colon\Pi\dcrosstk\cL&\to (\Pi\dcrossto\cL)\times_B (\Pi\dcrossok\cL)
      \end{align*}
      are algebra maps.
    \end{enumerate}
  In particular, $\Pi\dcrosskt\cL$ is a left $\Pi\dcrosskk\cL$-comodule algebra and a right $\Pi\dcrosstt\cL$-comodule algebra.
\end{lem}

\begin{lem}\label{lem. 2 cocycle of skew pairing}
    Let $\tau[\bullet|\bullet]$ and $\kappa[\bullet|\bullet]$ be two skew pairings between left bialgebroids $\Pi$ and $\cL$, then $ \Gamma\in {}_{B^{e}}\Hom(\Pi\dcrosskk\cL \otimes_{B^e}\Pi\dcrosskk\cL, B)$ given by
    \begin{align}    \Gamma((\alpha\dcrosskk{} X), (\beta\dcrosskk{} Y))=&\varepsilon((\alpha\dcrosstk{} X)\cdot_{\tau\kappa} (\beta\dcrosstk{} Y))
    \end{align}
    is a 2-cocycle, with inverse given by
     \begin{align}
        \Gamma^{-1}((\alpha\dcrosskk{} X), (\beta\dcrosskk{} Y))=&\varepsilon((\alpha\dcrosskt{} X)\cdot_{\kappa\tau} (\beta\dcrosskt{} Y)).
    \end{align}
    Moreover, $\Pi\dcrosstt\cL$ is a Drinfeld cotwist via $\Gamma$, i.e.
  $\Pi\dcrosstt\cL=(\Pi\dcrosskk\cL)^{\Gamma}$, and $\Pi\dcrosstk\cL$ is the one sided twist in the sense of (\ref{equ. left side twisted product}).
\end{lem}
\begin{proof}
 Let $T=(\alpha \dcrosstk X), V=(\beta \dcrosstk Y)$ and $W=\gamma \dcrosstk Z$, we have

 \begin{align*}
   V\cdot_{\tau\kappa} W=&\varepsilon((V\cdot_{\tau\kappa} W)\o) (V\cdot_{\tau\kappa} W)\t=\varepsilon(V\o \cdot_{\tau\kappa} W\o) V\t \cdot_{\kappa\kappa} W\t \\
 =&\Gamma(V\o, W\o) V\t \cdot_{\kappa\kappa} W\t,
 \end{align*}
 where the 2nd step uses the fact that $\Delta_{\tau\kappa\kappa}$ is an algebra map.
 Since the algebra $\Pi\dcrosstk\cL$ is associative \cite{schau1}, we can see $\Gamma$ is a 2-cocycle. Indeed, on the one hand,

 \begin{align*}
     \varepsilon(T\cdot_{\tau\kappa}(V \cdot_{\tau\kappa} W))=\Gamma(T, (V \cdot_{\tau\kappa} W))=\Gamma(T, \Gamma(V\o, W\o) V\t\cdot_{\kappa\kappa}W\t),
 \end{align*}
on the other hand,

\begin{align*}
    \varepsilon((T\cdot_{\tau\kappa}V )\cdot_{\tau\kappa} W)=\Gamma((T\cdot_{\tau\kappa}V), W)=\Gamma( \Gamma(T\o, V\o) T\t\cdot_{\kappa\kappa}V\t, W).
\end{align*}
In other words, $\Pi \dcrosstk \cL$ is a left hand twist of $\Pi \dcrosskk \cL$ in the sense of (\ref{equ. left side twisted product}). Since $\Delta_{\tau\omega\kappa}$ is an algebra map, so $(\Delta_{\tau\kappa\kappa}\ot \id)\circ \Delta_{\tau\kappa\tau}$ is an algebra as well. Therefore, we have
\begin{align*}
    V\cdot_{\tau\tau}W=&\varepsilon((V\cdot_{\tau\tau}W)\o)\overline{\varepsilon((V\cdot_{\tau\tau}W)\th)}(V\cdot_{\tau\tau}W)\t\\
    =&\varepsilon(V\o\cdot_{\tau\kappa}W\o)\overline{\varepsilon(V\th\cdot_{\kappa\tau}W\th)}V\t\cdot_{\kappa\kappa}W\t\\
    =&\Gamma(V\o, W\o)\overline{\Gamma^{-1}(V\th, W\th)}V\t\cdot_{\kappa\kappa}W\t.
\end{align*}
Moreover, by applying $\varepsilon$ on both sides of the above equality, we have
\begin{align*}
    \varepsilon(V\cdot_{\tau\tau}W)=&\varepsilon(\Gamma(V\o, W\o)\overline{\Gamma^{-1}(V\th, W\th)}V\t\cdot_{\kappa\kappa}W\t)\\
    =&\Gamma(V\o, W\o)\varepsilon(V\t\cdot_{\kappa\kappa}W\t)\Gamma^{-1}(V\th, W\th)\\
    =&\Gamma(V\o, W\o)\varepsilon(V\t \varepsilon(W\t))\Gamma^{-1}(V\th, W\th)\\
    =&\Gamma(V\o, \overline{\varepsilon(W\t)}W\o)\varepsilon(V\t )\Gamma^{-1}(V\th, W\th)\\
    =&\Gamma(V\o, W\o)\Gamma^{-1}(\varepsilon(V\t )V\th, W\t)\\
    =&\Gamma(V\o, W\o)\Gamma^{-1}(V\t, W\t).
\end{align*}
Therefore, we show $\Pi\dcrosstt$ is a Drinfeld twist of $\Pi\dcrosskk \cL$.

\end{proof}

From the above Lemma, we know $\Pi\dcrosstt\cL=(\Pi\dcrosskk\cL)^{\Gamma}$, and $\Pi\dcrosstk\cL$ is the left hand twist via $\Gamma$. Since $\overline{B}\subseteq\Pi\dcrosskk\cL$ is a left $\Pi\dcrosskk\cL$-Galois extension by its regular left coaction, so
by theorem \ref{thm. twi have one sted left Hopf Galois extensions}, we have the following result:

\begin{cor}
    Let $\tau[\bullet|\bullet]$ and $\kappa[\bullet|\bullet]$ be two shew pairing between two Hopf algebroids $\Pi$ and $\cL$, then
 $\BB\subseteq\Pi\dcrosstk \cL$ with the product (\ref{equ. twist on two side products}) is a left $\Pi\dcrosstt \cL$-Galois extension.
\end{cor}

\section{Cleft extensions and twisted crossed products of Hopf algebroids}

In this section, we will study cleft extensions and twisted crossed products of Hopf algebroids. Moreover, we are going to show cleft extensions, twisted crossed products and Hopf Galois extensions with normal basis properties are equivalent. Finally, we will also study the equivalence classes of cleft extensions.

\subsection{Cleft extensions of Hopf algebroids}

\begin{defi}\label{def. cleft extension}
    Let $\cL$ be a left bialgebroid over $B$. A $\cL$-cleft extension consists of a left $\cL$-comodule algebra $P$ and an unital left $\cL$-comodule map $\gamma:  \cL\to P$ (called cleaving map), such that
    \begin{itemize}
        \item $\BB\subseteq N=P^{co\cL}$ (thus $P$ is a $\BB$-ring), and $\gamma$ is  $\BB$-bilinear, i.e. $\gamma(\Bar{a}X\Bar{b})=\Bar{a}\gamma(X)\Bar{b}$;
        \item $j: \cL\ot_{\BB}P\to \cL\diamond_{B}P$ given by $j(X\ot_{\BB} p)=X\o\di \gamma(X\t)p$ is bijective.
    \end{itemize}
\end{defi}
\begin{rem}
 Since $\gamma$ is right $\BB$-linear, we can see $j$ factors through $\cL\ot_{\BB}P$. Moreover,  since $\gamma$ is a left $\cL$-colinear, we have $\gamma(bXb')=b\gamma(X)b'$ for any $b, b'\in B$ and $X\in \cL$. We can also see the image of $j$ belongs to $\cL\diamond_{B} P$. Therefore, $j:\cL\ot_{\BB}P\to \cL\diamond_{B}P$ is well defined.
\end{rem}
If we denote $j^{-1}|_{\cL\di 1_{P}}: X\mapsto X^{\alpha}\ot_{\BB}X^{\beta}\in \cL\ot_{\BB}P$, we have

\begin{prop}
    Let $\cL$
be a left Hopf algebroid over $B$ and $N\subseteq P$ be a $\cL$-cleft extension with cleaving map $\gamma$, then we have
    \begin{align}\label{equ. cleaving map 1}
X^{\alpha}\o\di\gamma(X^{\alpha}\t)X^{\beta}=&X\di 1,\\
\label{equ. cleaving map 2}
X\o{}^{\alpha}\ot_{\BB} X\o{}^{\beta}\gamma(X\t)=&X\ot_{\BB}1,\\
\label{equ. cleaving map 5}
(a\Bar{a'}X b\Bar{b'})^{\alpha}\ot_{\BB}(a\Bar{a'}X b\Bar{b'})^{\beta}=&a X^{\alpha}b\ot_{\BB}b' X^{\beta} a',\\
\label{equ. cleaving map 6}
\Bar{b}X^{\alpha}\ot_{\BB}X^{\beta}=&X^{\alpha}\ot_{\BB}X^{\beta}\Bar{b},\\
\label{equ. cleaving map 3}
(X^{\alpha})\o\di(X^{\alpha})\t\ot_{\BB}X^{\beta}=&X\o\di(X\t)^{\alpha} \ot_{\BB} (X\t)^{\beta},\\
\label{equ. cleaving map 4}
X^{\alpha}\ot X^{\beta}\mo\ot X^{\beta}\z=&(X_{+})^{\alpha}\ot X_{-}\ot (X_{+})^{\beta}\in\int^{c,d} \int_{a,b}{}_{\Bar{c}}\cL_{\Bar{a}}\ot {}_{\Bar{b}}\cL_{\Bar{d}}\ot {}_{\Bar{a}, b}P_{\Bar{c},d},
\end{align}
for any $a, a', b, b'\in B$, $X\in \cL$ and $p\in P$.
\end{prop}

\begin{proof}
    It is easy to see (\ref{equ. cleaving map 1}) and (\ref{equ. cleaving map 2}) are the direct results of the definition. We can see (\ref{equ. cleaving map 3})  can be proved by comparing the result of $\id_{\cL}\di j$ on both sides. We can see (\ref{equ. cleaving map 5}) and \ref{equ. cleaving map 6}  can be proved by comparing the results of $j$ on both sides. For (\ref{equ. cleaving map 4}), we first define a map $\phi:=(\lambda\ot \id_{P})\circ j_{1,3}: \int_{a,b}\cL_{\Bar{a}}\ot {}_{\Bar{b}}\cL\ot {}_{\Bar{a}, b}P\to \cL\di \cL\di P$, where $j_{1,3}:\int_{a,b}\cL_{\Bar{a}}\ot {}_{\Bar{b}}\cL\ot {}_{\Bar{a}, b}P\to \int_{a,b}  {}_{\Bar{b}}\cL_{\Bar{a}}\ot {}_{\Bar{a}}\cL\ot {}_{b}P
    $ is given by
    \[j_{1,3}(X\ot Y\ot p)=X\o\ot Y\ot \gamma(X\t)p,\]
     for any $X, Y\in \cL$ and $p\in P$. We can see $j_{1,3}$ is well defined since
     \begin{align*}
         j_{1,3}(X\ot \Bar{b}Y\ot p)=&X\o\ot \Bar{b}Y\ot \gamma(X\t)p=X\o\Bar{b}\ot Y\ot \gamma(X\t)p\\
         =&X\o\ot Y\ot \gamma(X\t b)p=X\o\ot Y\ot \gamma(X\t )bp.
     \end{align*}
    It is not hard to see $\phi$ is well defined and has a more precisely formula,
    \[\phi(X\ot Y\ot p)=X\o\ot X\t Y\ot \gamma(X\th)p.\]
     We can see $\phi$ is invertible with
    \[\phi^{-1}(X\ot Y\ot p)=(X_{+})^{\alpha}\ot X_{-}Y \ot (X_{+})^{\beta}p,\]
    which is also well-defined. Indeed, let $X\ot Y\ot p\in \cL\di\cL\di P$
    \begin{align*}
        \phi^{-1}(\Bar{b}X\ot Y\ot p)=(X_{+})^{\alpha}\ot X_{-}bY \ot (X_{+})^{\beta}p=\phi^{-1}(X\ot bY\ot p),
    \end{align*}
    and
    \begin{align*}
        \phi^{-1}(X\ot \Bar{b}Y\ot p)=&(X_{+})^{\alpha}\ot X_{-}\Bar{b} Y \ot (X_{+})^{\beta}p=(\Bar{b}X_{+})^{\alpha}\ot X_{-} Y \ot (\Bar{b}X_{+})^{\beta}p\\
        =&(X_{+})^{\alpha}\ot X_{-}Y \ot (X_{+})^{\beta}bp=\phi^{-1}(X\ot Y\ot bp).
    \end{align*}
    Moreover,
    \begin{align*}
        \phi\circ\phi^{-1}(X\ot Y\ot p)=&(X_{+})^{\alpha}\o\o\ot (X_{+})^{\alpha}\o\t X_{-}Y \ot \gamma((X_{+})^{\alpha}\t)(X_{+})^{\beta}p\\
        =&X_{+}\o\ot X_{+}\t X_{-}Y \ot p\\
       =&X\ot Y\ot P,
    \end{align*}
    and
    \begin{align*}
        \phi^{-1}\circ \phi(X\ot Y\ot p)=&X\o{}_{+}{}^{\alpha}\ot X\o{}_{-} X\t Y\ot X\o{}_{+}{}^{\beta}\gamma(X\th)p\\
        =&X\o{}^{\alpha}\ot Y\ot X\o{}^{\beta}\gamma(X\t)p\\
        =&X\ot Y\ot p.
    \end{align*}
    By applying $\phi$ on both side of (\ref{equ. cleaving map 4}), we have on the one hand,
    \begin{align*}
        \phi(X^{\alpha}\ot X^{\beta}\mo\ot X^{\beta}\z)=&X^{\alpha}\o\ot X^{\alpha}\t X^{\beta}\mo\ot \gamma(X^{\alpha}\th)X^{\beta}\z\\
        =&X^{\alpha}\o\ot (\gamma(X^{\alpha}\t) X^{\beta})\mo\ot (\gamma(X^{\alpha}\t) X^{\beta})\z\\
        =&X\ot 1\ot 1,
    \end{align*}
    where the second step uses $\gamma$ is left $\cL$-colinear. On the other hand, since the right hand side of (\ref{equ. cleaving map 4}) is $\phi^{-1}(X\ot 1\ot 1)$, so $\phi((X_{+})^{\alpha}\ot X_{-}\ot (X_{+})^{\beta})=X\ot 1\ot 1$.
\end{proof}

\begin{prop}\label{prop. cleft extension}
    Let $\cL$
be a left Hopf algebroid over $B$ and $N\subseteq P$ be a $\cL$-cleft extension with cleaving map $\gamma$, then
    \[p\mo{}^{\alpha}\ot_{\BB}p\mo{}^{\beta} p\z\in \cL\ot_{\BB}N\]
    for any $p\in P$.
\end{prop}
\begin{proof}
 We can observe that in the following diagram both the diagram without dash arrows and the diagram with dash arrows commute

 \[
\begin{tikzcd}
  &\cL\ot_{\BB} P \arrow[d, "\cL\ot {}_{L}\delta" left ]\arrow[d, dashrightarrow,"i_{1,3}" right, shift left=1ex] \arrow[r, "j"] & \cL \diamond_{B} P  \arrow[dd, "\cL\ot {}_{L}\delta" left ]\arrow[dd,dashrightarrow, "\Delta\ot_{P}" right, shift left=1ex] &\\
  &\int_{c,d}\cL_{\Bar{c}}\ot{}_{\Bar{d}}\cL\ot {}_{\Bar{c}, d}P \arrow[d, " j_{1,3}"]\quad&\qquad&\\
   &\int_{c,d} {}_{\Bar{c}}\cL_{\Bar{d}}\ot{}_{\Bar{d}}\cL\ot {}_{c}P   \arrow[r, "\lambda\ot P"] & \cL\diamond_{B} \cL \diamond_{B} P, &
\end{tikzcd}
\]
where $i_{1, 3}:\cL\ot_{\BB}P\to \int_{c,d}\cL_{\Bar{c}}\ot{}_{\Bar{d}}\cL\ot {}_{\Bar{c}, d}P$ is given by $i_{1.3}(X\ot p)=X\ot 1\ot p$ for any $X\ot p\in \cL\ot_{\BB}P$, and $j_{1, 3}:\int_{c,d}\cL_{\Bar{c}}\ot{}_{\Bar{d}}\cL\ot {}_{\Bar{c}, d}P \to \int_{c,d} {}_{\Bar{c}}\cL_{\Bar{d}}\ot{}_{\Bar{d}}\cL\ot {}_{c}P$ is given by $j_{1,3}(X\ot Y\ot p)=X\o\ot Y\ot \gamma(X\t)p$.
Since $\gamma$ is $B$-bilinear, we can see $j_{1,3}$ is well defined. Indeed, $j_{1,3}(X\ot \Bar{b}Y\ot p)=X\o \ot \Bar{b}Y\ot \gamma(X\t)p=X\o \Bar{b}\ot Y\ot \gamma(X\t)p=X\o \ot Y\ot \gamma(X\t b)p=X\o \ot Y\ot \gamma(X\t)bp.$

Since $j_{1,3}$ and $\lambda\ot P$ are injective, $\cL\ot_{\BB}N$ is the coequalizer of the `left-down' morphisms from $\cL\ot_{\BB}P$ to $\cL\diamond_{B}\cL\diamond_{B}P$. So to check  $p\mo{}^{\alpha}\ot_{\BB}p\mo^{\beta} p\z\in \cL\ot_{\BB}N$, it is sufficient to check $p\mo{}^{\alpha}\ot_{\BB}p\mo^{\beta} p\z$ belongs to the coequalizer of the `up-right' morphisms from $\cL\ot_{\BB}P$ to $\cL\di{}\cL\di{}P$. Indeed,
\[(\cL\ot {}_{L}\delta)\circ j(p\mo{}^{\alpha}\ot_{\BB}p\mo^{\beta} p\z)=p\mt\ot p\mo\ot p\z\
=(\Delta\ot P)\circ j(p\mo{}^{\alpha}\ot_{\BB}p\mo^{\beta} p\z).\]
\end{proof}

\begin{lem}\label{lem. cleft extension is Galois}
   Let $\cL$
be a left Hopf algebroid over $B$, if $N\subseteq P$ is a $\cL$-cleft extension with cleaving map $\gamma$, then $N\subseteq P$ is a left $\cL$-Galois extension. More precisely, the left translation map is given by
    \begin{align}
        X\tuno{}\ot_{N}X\tdue{}=\gamma(X^{\alpha})\ot_{N}X^{\beta},
    \end{align}
    for any $X\in \cL$.
\end{lem}
\begin{proof}
    First, we can see the left translation map is well defined since $\gamma$ is $\BB$-bilinear.
    Then we can see
    \begin{align*}
        {}^{L}\can\circ {}^{L}\tau (X)=&\gamma(X^{\alpha})\mo\di{}\gamma(X^{\alpha})\z X^{\beta}\\
        =&X^{\alpha}\o\di{}\gamma(X^{\alpha}\t) X^{\beta}\\
        =&X\di{}1.
    \end{align*}
   We can also see
     \begin{align*}
         {}^{L}\can^{-1}\circ {}^{L}\can(p\ot 1)=&\gamma(p\mo{}^{\alpha})\ot_{N}p\mo{}^{\beta}p\z=\gamma(p\mo{}^{\alpha})p\mo{}^{\beta}p\z\ot_{N} 1\\
         =&\varepsilon(p\mo{}^{\alpha}\o)\gamma(p\mo{}^{\alpha}\t)p\mo{}^{\beta}p\z\ot_{N} 1\\
         =&\varepsilon(p\mo)p\z\ot_{N} 1=p\ot_{N}1,
     \end{align*}
     where we use Proposition \ref{prop. cleft extension} in the 2nd step.
\end{proof}

\begin{prop}\label{prop. algebra structure given by cleft extension}
     Let $\cL$
be a left Hopf algebroid over $B$, if $N\subseteq P$ is a $\cL$-cleft extension with cleaving map $\gamma$, then $\cL\ot_{\BB}N$ is an algebra with the product given by
\[(X\ot n)(Y\ot m)=(X\o Y\o)^{\alpha}\ot (X\o Y\o)^{\beta}\gamma(X\t)n\gamma(Y\t)m,\]
for any $X\ot n, Y\ot m\in \cL\ot_{\BB} N$. We denote this algebra by $\cL\#^{\gamma} N$.
\end{prop}
\begin{proof}

First, we can define a map $\phi: (\cL\times P)\ot(\cL\times P)\to \cL\times P$ by
\[\phi((X\ot p)\ot(Y\ot q)):=XY\ot pq,\]
for any $X\ot p, Y\ot q\in \cL\times P$. Since $B$ commutes with $N$ as subalgebras of $P$, we can also see $j(X\ot_{\BB}n)=X\o\ot \gamma(X\t)n\in \cL\times_{B}P$. As a result, $(X\ot n)(Y\ot m)=j^{-1}\circ \phi\circ (j\ot j)((X\ot n)\ot(Y\ot m))$ is well defined. Now, let's check the product is associative. Let $X\ot n, Y\ot m, Z\ot l\in \cL\ot_{\BB}N$, on the one hand,
\begin{align*}
    (&(X\ot n)(Y\ot n))(Z\ot l)\\
    =&((X\o Y\o)^{\alpha}\o Z\o)^{\alpha}\ot ((X\o Y\o)^{\alpha}\o Z\o)^{\beta}\gamma((X\o Y\o)^{\alpha}\t)(X\o Y\o)^{\beta}\gamma(X\t)n\gamma(Y\t)m\gamma(Z\t)l\\
    =&(X\o Y\o Z\o)^{\alpha}\ot (X\o Y\o Z\o)^{\beta}\gamma(X\t)n\gamma(Y\t)m\gamma(Z\t)l.
\end{align*}
   On the other hand,
   \begin{align*}
      (&X\ot n)((Y\ot n)(Z\ot l))\\
      =&(X\o (Y\o Z\o)^{\alpha}\o)^{\alpha}\ot (X\o (Y\o Z\o)^{\alpha}\o)^{\beta} \gamma(X\t) n \gamma((Y\o Z\o)^{\alpha}\t)(Y\o Z\o)^{\beta}\gamma(Y\t)m\gamma(Z\t)l\\
      =&(X\o Y\o Z\o)^{\alpha}\ot (X\o Y\o Z\o)^{\beta}\gamma(X\t)n\gamma(Y\t)m\gamma(Z\t)l.
   \end{align*}
\end{proof}

\begin{lem}\label{lem. cleft extension has normal basis property}
     Let $\cL$
be a left Hopf algebroid over $B$, if $N\subseteq P$ is a $\cL$-cleft extension with cleaving map $\gamma$, then $\cL\#^{\gamma}N$ is isomorphic to $P$ as a $\cL$-comodule algebra.
\end{lem}
\begin{proof}
    We can see $\cL\#^{\gamma}N$ is a left $\cL$-comodule with the $B$-bimodule structure and coaction given by
    \[b.(X\ot n).b'=bXb'\ot n,\quad\delta(X\ot n)=X\o\ot X\t\ot n,\]
    for any $X\ot n\in \cL\#_{\gamma}N$ and $b,b'\in B$. We can also see this coaction is an algebra map
    \begin{align*}
        \delta((X\ot n)(Y\ot m))=&(X\o Y\o)^{\alpha}\o\ot (X\o Y\o)^{\alpha}\t \ot (X\o Y\o)^{\beta}\gamma(X\t)n\gamma(Y\t)m\\
        =&X\o Y\o\ot (X\t Y\t)^{\alpha}\ot (X\t Y\t)^{\beta}\gamma(X\th)n\gamma(Y\th)m\\
        =&\delta(X\ot n)\delta(Y\ot m),
    \end{align*}
    where the 2nd step uses (\ref{equ. cleaving map 3}). The isomorphism between $P$ and $\cL\#^{\gamma}N$ is given by Proposition \ref{prop. cleft extension}. Namely,
    \[\phi: P\to \cL\#_{\gamma}N, \quad p\mapsto p\mo{}^{\alpha}\ot_{\BB}p\mo{}^{\beta} p\z,\]
    for any $p\in P$. We can see $\phi$ is invertible with inverse $\phi^{-1}: \cL\#_{\gamma}N\to P$ given by $\phi^{-1}(X\ot n)=\gamma(X)n$, for any $X\ot n\in \cL\#_{\gamma}N$. We can see $\phi^{-1}$ is well defined. For any $X\ot n\in \cL\#_{\gamma}N$,
    \begin{align*}
        \phi\circ \phi^{-1}(X\ot n)=&\gamma(X)\mo{}^{\alpha}\ot \gamma(X)\mo{}^{\beta}\gamma(X)\z n
        = X\o{}^{\alpha}\ot X\o{}^{\beta}\gamma(X\t) n\\
        =&X\ot n,
    \end{align*}
 where the 2nd step uses the fact that $\gamma$ is left $\cL$-colinear, the 3rd step uses (\ref{equ. cleaving map 2}). Also,
 \begin{align*}
     \phi^{-1}\circ \phi(p)=\gamma(p\mo{}^{\alpha})p\mo{}^{\beta}p\z=\varepsilon(p\mo{}^{\alpha}\o)\gamma(p\mo{}^{\alpha}\t)p\mo{}^{\beta}p\z
     =\varepsilon(p\mo)p\z
     =p,
 \end{align*}
 where the 2nd step uses (\ref{equ. cleaving map 1}). By (\ref{equ. cleaving map 5}), we can see $\phi$ is $B$-bilinear. We can check $\phi$ is $\cL$-colinear:
 \begin{align*}
     \delta\circ \phi(p)=&p\mo{}^{\alpha}\o\ot p\mo{}^{\alpha}\t \ot p\mo{}^{\beta}p\z
     =p\mt\ot p\mo{}^{\alpha}\ot p\mo{}^{\beta}p\z\\
     =&(\id\ot \phi)\circ \delta(p),
 \end{align*}
where the 2nd step uses (\ref{equ. cleaving map 3}). Finally, we can check $\phi$ is an algebra map:
\begin{align*}
    \phi(p)\phi(q)=&(p\mo{}^{\alpha}\ot p\mo{}^{\beta}p\z)(q\mo{}^{\alpha}\ot q\mo{}^{\beta}q\z)\\
    =&(p\mo{}^{\alpha}\o q\mo{}^{\alpha}\o)^{\alpha}\ot (p\mo{}^{\alpha}\o q\mo{}^{\alpha}\o)^{\beta} \gamma(p\mo{}^{\alpha}\t)p\mo{}^{\beta}p\z \gamma(q\mo{}^{\alpha}\t)q\mo{}^{\beta}q\z\\
    =&(p\mo q\mo)^{\alpha}\ot (p\mo q\mo)^{\beta} p\z q\z\\
    =&\phi(pq),
\end{align*}
for any $p, q\in P$, where the 3rd step uses (\ref{equ. cleaving map 1}).
\end{proof}

\begin{defi}
    Let $\cL$ be a left bialgebroid over $B$ and $P$ be a left $\cL$-comodule algebra with $N={}^{co\cL}P$, we say $P$ has the left normal basis property if $\overline{B}\subseteq N$ and $P\simeq \cL\ot_{\BB}N$ as left $\cL$-comodules, left $\Bar{B}$-modules (with the left $\BB$-module structure given by $\Bar{b}.(X\ot_{\BB}n)=\Bar{b}X\ot_{\BB}n$, for any $X\ot_{\BB}n\in \cL\ot_{\BB}N$ and $\Bar{b}\in \Bar{B}$) and right $N$-modules (with its natural left $\cL$-comodule and right $N$-module structures).
\end{defi}

\begin{thm}
     Let $\cL$ be a left Hopf algebroid over $B$ and $P$ be a left $\cL$-comodule algebra with $N={}^{co\cL}P$, then $N\subseteq P$ is a cleft extension if and only if $N\subseteq P$ is a left $\cL$-Galois extension and $P$ has the left normal basis property.
\end{thm}
\begin{proof}
    By Lemmas \ref{lem. cleft extension is Galois} and \ref{lem. cleft extension has normal basis property}, we know if $N\subseteq P$ is a cleft extension, then $N\subseteq P$ is a left $\cL$-Galois extension and $P$ has the left normal basis property. Conversely, we assume $\phi:P\to \cL\ot_{\BB}N$ is an isomorphism of left $\cL$-comodules, left $\BB$-modules and right $N$-modules. Define $\gamma:\cL\to P$ by $\gamma(X):=\phi^{-1}(X\ot 1)$. Since $\phi^{-1}$ is left $\cL$-colinear,  $\gamma$ is left $\cL$-colinear as well. Because $\phi^{-1}$ is right $N$-linear and $\BB\subseteq N$, $\gamma(X\Bar{b})=\phi^{-1}(X\ot_{\BB}\Bar{b})=\phi^{-1}(X\ot_{\BB}1)\Bar{b}=\gamma(X)\Bar{b}$. We can also see $\gamma$ is left $\BB$-linear since $\phi$ is. Now, let's show $j: \cL\ot_{\BB}P\to \cL\di P$ is bijective with its inverse $j^{-1}$ determined by
    \[j^{-1}(X\ot 1)=X^{\alpha}\ot X^{\beta}:=\phi(X\tuno{})X\tdue{},\]
    for any $X\in \cL$. In other words, if we denote the image of $\phi$ by $\phi(p)=:L(p)\ot R(p)\in \cL\ot_{\BB} N$, we have $X^{\alpha}\ot X^{\beta}=L(X\tuno{})\ot_{\BB}R(X\tuno{})X\tdue{}$. Since $\phi$ is right $N$-linear, $\phi(X\tuno{})X\tdue{}$ is well defined. Moreover, we have on the one hand,
    \begin{align*}
       j(L(X\tuno{})&\ot_{\BB}R(X\tuno{})X\tdue{}) \\=&L(X\tuno{})\o\ot\gamma(L(X\tuno{})\t)R(X\tuno{})X\tdue{}\\
        =&X\tuno{}\mo\ot \gamma(L(X\tuno{}\z))R(X\tuno{}\z)X\tdue{}\\
        =&X\tuno{}\mo\ot \phi^{-1}(L(X\tuno{}\z)\ot 1)R(X\tuno{}\z)X\tdue{}\\
        =&X\tuno{}\mo\ot \phi^{-1}(L(X\tuno{}\z)\ot R(X\tuno{}\z))X\tdue{}\\
        =&X\tuno{}\mo\ot \phi^{-1}(\phi(X\tuno{}\z))X\tdue{}\\
        =&X\tuno{}\mo\ot X\tuno{}\z X\tdue{}\\
        =&X\ot 1,
    \end{align*}
    where the 2nd step uses the fact that $\phi$ is left $\cL$-colinear, the 4th step uses the fact that $\phi^{-1}$ is right $N$-linear. On the other hand,
    \begin{align*}
        L(X\o\tuno{})&\ot R(X\o\tuno{})X\o\tdue{}\gamma(X\t)\\
        =&L(\gamma(X)\mo\tuno{})\ot R(\gamma(X)\mo\tuno{}) \gamma(X)\mo\tdue{} \gamma(X)\z\\
        =&L(\gamma(X))\ot R(\gamma(X))
        =\phi(\gamma(X))
        =\phi\circ \phi^{-1}(X\ot 1)\\
        =&X\ot 1.
    \end{align*}
    where the 1st step uses the fact that $\gamma$ is left $\cL$-colinear.
\end{proof}

\subsection{Twisted crossed products of Hopf algebroids}

\begin{defi}\label{defi. 2 cocycle and twisted module algebra}
    Let $\cL$ be a left bialgebroid over $B$ and $M$ be a $B$-ring. We call $\cL$ measures $M$, if there is a linear map (called measuring) $\la: \cL\ot_{B}{} M\to M$, $X\ot m\mapsto X\la m$, such that
    \begin{itemize}
        \item [(1)] $(b \Bar{b'}X)\la m=b(X\la m)b'$,
        \item[(2)] $X\la b=\varepsilon(Xb)=\varepsilon(X\Bar{b})$,
        \item[(3)] $X\la (mn)=(X\o\la m)(X\t \la n)$,
    \end{itemize}
    for any $m,n\in M$, $b, b'\in B$ and $X\in \cL$.
    Given such a measuring $\la$, a convolution invertible map $\sigma: \cL\ot_{B^{e}}\cL\to M$ is called a 2-cocycle of $\cL$ over $M$, if it satisfies:
    \begin{itemize}
        \item [(i)] $\sigma(b\Bar{b'}X, Y)=b\sigma(X, Y)b'$,
        \item[(ii)] $\sigma(X\ot_{B^{e}}1)=\sigma(1\ot_{B^{e}}X)=\varepsilon(X)$,
        \item[(iii)] $\sigma(X, Y\Bar{b})=\sigma(X, Yb)$,
        \item[(iv)] $\sigma(X\o, Y\o Z\o)(X\t\la \sigma(Y\t, Z\t))=\sigma(X\o Y\o, Z)\sigma(X\t, Y\t)$,
    \end{itemize}
    for any $X, Y, Z\in \cL$. Given such a measuring $\la$ and 2-cocycle $\sigma$, $M$ is called a $\sigma$-twisted left $\cL$-module algebra, if
    \begin{itemize}
        \item [(v)] $\sigma(X\o, Y\o) (X\t\la(Y\t \la m))=((X\o Y\o)\la m)\sigma(X\t, Y\t)$.
    \end{itemize}
\end{defi}

\begin{rem}
    By (1) and (i), we can see (3), (iv) and (v) are well defined. Moreover, if $M=B$, a 2-cocycle of $\cL$ over $B$ is a right-handed 2-cocycle given in Section \ref{sec. Drinfeld twist of Hopf algebroids}, with the measuring $\la$ given by $X\la b=\varepsilon(Xb)=\varepsilon(X\Bar{b})$. Indeed, the left hand of (iv) is
    \begin{align*}
       \sigma(X\o, Y\o Z\o)\varepsilon(X\t \sigma(Y\t, Z\t))=&\sigma(X\o \overline{ \sigma(Y\t, Z\t)}, Y\o Z\o)\varepsilon(X\t)\\
    =&\sigma(X\o, \overline{ \sigma(Y\t, Z\t)} Y\o Z\o)\varepsilon(X\t)\\
    =&\sigma(X, \overline{ \sigma(Y\t, Z\t)} Y\o Z\o),
    \end{align*}
 and  the right hand of (iv) is $\sigma(\overline{\sigma(X\t, Y\t)}X\o Y\o, Z)$. For the same reason as in \cite{HM22} and \cite{HM23}, we can see condition (iii) and (iv) implies $\sigma^{-1}$ is a `left-handed' 2-cocycle, namely, $\sigma^{-1}$ satisfies
 \[X\o\la \sigma^{-1}(Y\o, Z\o)\sigma^{-1}(X\t, Y\t Z\t)=\sigma^{-1}(X\o, Y\o)\sigma^{-1}(X\t Y\t, Z),\]
and (i), (ii), (iii). When $M=B$, $\sigma^{-1}$ is a 2-cocycle over $B$ in the sense of Definition \ref{def. left handed 2 cocycle}.
\end{rem}

Given a left $B$-bialgebroid $\cL$, the coopposite $B$-coring $\cL^{cop}$ is a left bialgebroid over $\BB$ with the source and target map given by $s^{cop}:=t: \BB\to \cL$ and $t^{cop}:=s: B\to \cL$. In the following, we will be more interested in  2-cocycles of $\cL^{cop}$ over a $\overline{B}$-ring, and we first have a useful proposition below:

\begin{prop}\label{prop. proposition of 2-cocycle over a ring}
    Let $\cL$ be a left Hopf algebroid over $B$ and $N$ be a $\BB$-ring, if $(\sigma, \la)$ is a 2-cocycle of $\cL^{cop}$ over $N$, then we have
    \[\sigma(X\t{}_{+}, X\t{}_{-}\t)(X\o\la \sigma^{-1}(X\t{}_{-}\o, X\th))=\varepsilon(X),\]
    for any $X\in \cL$.
\end{prop}

\begin{proof}
    We first show the formula above factors through all the balanced tensor products. Since for any $X\in \cL$,
    \[X_{+}\ot X_{-}\o\ot X_{-}\t\in \int^{a,b}\int_{c,d}{}_{\Bar{a}}\cL_{\Bar{c}}\ot {}_{ \Bar{d}}\cL_{\Bar{b}}\ot {}_{\Bar{c}, d}\cL_{\Bar{a},b},\]
    we have
    \[X\o\ot X\t{}_{+}\ot X\t{}_{-}\o\ot X\t{}_{-}\t\ot X\th{}\in \int^{a,b,e,f}\int_{c,d,g,h}{}_{\Bar{g}}\cL_{\Bar{e}}\ot {}_{\Bar{a},g}\cL_{\Bar{c},e}\ot {}_{ \Bar{d}, f}\cL_{\Bar{b}, h}\ot {}_{\Bar{c}, d}\cL_{\Bar{a},b}\ot {}_{h}\cL_{f} \]
    for any $X\in \cL$. Let $X\ot Y\ot Z\ot W\ot V\in \int^{a,b,e,f}\int_{c,d,g,h}{}_{\Bar{g}}\cL_{\Bar{e}}\ot {}_{\Bar{a},g}\cL_{\Bar{c},e}\ot {}_{ \Bar{d}, f}\cL_{\Bar{b}, h}\ot {}_{\Bar{c}, d}\cL_{\Bar{a},b}\ot {}_{h}\cL_{f}$. We can see
    \begin{align*}
        \sigma(Y, W)(\Bar{b}X\la \sigma^{-1}(Z, V))=\sigma(Y, W)\Bar{b}(X\la \sigma^{-1}(Z, V))
        =\sigma(bY, W)X\la \sigma^{-1}(Z, V)),
    \end{align*}
    and
    \begin{align*}
       \sigma(Y\Bar{b}, W)(X\la \sigma^{-1}(Z, V))=\sigma(Y, \Bar{b}W)(X\la \sigma^{-1}(Z, V)),
    \end{align*}
    and
    \begin{align*}
        \sigma(Y, W)(X\la \sigma^{-1}(\Bar{b}Z, V))=&\sigma(Y, W)((X\Bar{b})\la \sigma^{-1}(Z, V))
        =\sigma(Yb, W)(X\la \sigma^{-1}(Z, V))\\
        =&\sigma(Y, bW)(X\la \sigma^{-1}(Z, V)),
    \end{align*}
    and
    \begin{align*}
        \sigma(Y, W)(X\la \sigma^{-1}(Zb, V))=&\sigma(Y, W)(X\la \sigma^{-1}(Z, bV)).
    \end{align*}
    So the formula is well defined. By (iv), we have
    \[X\la\sigma(Y, Z)=\sigma^{-1}(X\th, Y\th Z\t)\sigma(X\t Y\t, Z\o)\sigma(X\o, Y\o),\]
    for any $X, Y, Z\in\cL$. Therefore, by (iii) we can get
    \[X\la \sigma^{-1}(Y, Z)=\sigma^{-1}(X\th, Y\th)\sigma^{-1}(X\t Y\t, Z\t)\sigma(X\o, Y\o Z\o).\]
    As a result,
    \begin{align*}
        \sigma(&X\t{}_{+}, X\t{}_{-}\t)(X\o\la \sigma^{-1}(X\t{}_{-}\o, X\th))\\
        =&\sigma(X\t{}_{+}, X\t{}_{-}\t)\sigma^{-1}(X\o\th, X\t{}_{-}\o\th)\\
        &\sigma^{-1}(X\o\t X\t{}_{-}\o\t, X\th\t)\sigma(X\o\o, X\t{}_{-}\o\o X\th\o)\\
        =&\sigma(X\o{}_{+}\t, X\o{}_{-}\t)\sigma^{-1}(X\o{}_{+}\o\th, X\o{}_{-}\o\th)\\
        &\sigma^{-1}(X\o{}_{+}\o\t X\o{}_{-}\o\t, X\t\t)\sigma(X\o{}_{+}\o\o, X\o{}_{-}\o\o X\t\o)\\
        =&\sigma^{-1}(X\o{}_{+}\t X\o{}_{-}\t, X\t\t)\sigma(X\o{}_{+}\o, X\o{}_{-}\o X\t\o)\\
        =&\varepsilon(X).
    \end{align*}
\end{proof}

\begin{rem}\label{rem. classical explanation for cleft extension}

Recall that given a Hopf algebra $H$, a twisted crossed product $H\#_{\Tilde{\sigma}} M$ can be given by a $\Tilde{\sigma}$-twisted right $H$-modules algebra $M$, namely, there is a right measuring $\ra: M\ot H\to H$ and a 2-cocycle $\Tilde{\sigma}:H\ot H\to M$, with the product of $H\#_{\Tilde{\sigma}} M$ given by
\[(h\ot n)(g\ot m)=h\o g\o\ot \Tilde{\sigma}(h\t, g\t)(n\ra g\th)m,\]
for any $h\ot n, g\ot m\in H\#_{\Tilde{\sigma}}M$. However, when we consider a left bialgebroid $\cL$, there is no proper way to give the notion of right measuring.
So in order to define a twisted left $\cL$-module algebra $M$, we must turn the `right-handed' twisted-module algebra into the `left' one. In the classical case, given a Hopf algebra $H$ and its $\Tilde{\sigma}$-twisted right module algebra $M$, this could be done by defining a left measuring and a 2-cocycle by $h\la m:=m\ra S^{-1}(h)$,  $\sigma(h,g):=\Tilde{\sigma}(S^{-1}(g), S^{-1}(h))$, for any $m\in M$ and $h,g\in H$. The resulting 2-cocycle $(\la, \sigma)$ is a 2-cocycle of $H^{cop}$ and $M$ is a $\sigma$-twisted $H^{cop}$-module algebra that fits Definition \ref{defi. 2 cocycle and twisted module algebra}. As it will turn out, the left measuring and the two-cocycle $\tilde\sigma$ have analogs in the case where $H$ is generalized to a Hopf algebroid $\cL$, while the right measuring and cocycle $\sigma$ do not.

\end{rem}

\begin{lem}\label{lem. twisted module algebra and crossed product}
 Let $\cL$ be a Hopf algebroid over $B$ and $N$ be a $\BB$-ring. If $(\sigma, \la)$ is a pair of maps satisfy (1)-(3) and (i)-(iii) associated to $\cL^{cop}$ and $N$, then $\cL\ot_{\BB}N$ is an algebra with the product
 \begin{align}\label{equ. product of the twisted crossed product}
     (X\ot n)(Y\ot m)=X_{+}Y_{+}\ot \sigma(Y_{-}\t, X_{-})(Y_{-}\o\la n)m,
 \end{align}
 for any $X\ot n, Y\ot m\in \cL\ot_{\BB}N$ if and only if $(\sigma,\la)$ satisfies (iv) and $N$ is a $\sigma$-twisted left $\cL^{cop}$-module algebra. We call the algebra $\sigma$-twisted crossed product which is denoted by $\cL\#_{\sigma}N$.
\end{lem}

\begin{proof}
  If $N$ is a $\sigma$-twisted left $\cL$-module algebra, the product factors through the balanced tensor products.
    \begin{align*}
        (X\Bar{b}\ot n)(Y\ot m)=&X_{+}Y_{+}\ot \sigma(Y_{-}\t, bX_{-})(Y_{-}\o\la n)m\\
        =&X_{+}Y_{+}\ot \sigma(Y_{-}\t b, X_{-})(Y_{-}\o\la n)m\\
        =&X_{+}Y_{+}\ot \sigma(Y_{-}\t , X_{-})((Y_{-}\o \Bar{b})\la n)m\\
        =&X_{+}Y_{+}\ot \sigma(Y_{-}\t , X_{-})(Y_{-}\o\la (\Bar{b}n))m\\
        =&(X\ot \Bar{b}n)(Y\ot m)
    \end{align*}
Also
\begin{align*}
    (X\ot n)(Y\Bar{b} \ot m)=&X_{+}Y_{+}\ot \sigma(Y_{-}\t, X_{-})((bY_{-}\o)\la n)m\\
    =&X_{+}Y_{+}\ot \sigma(Y_{-}\t, X_{-})(Y_{-}\o\la n)\Bar{b}m\\
    =&(X\ot n)(Y\ot \Bar{b}m).
\end{align*}
Since $Y_{+}\ot Y_{-}\o\ot Y_{-}\t\in \int^{a,b}\int_{c,d}{}_{\Bar{a}}\cL_{\Bar{c}}\ot {}_{\Bar{d}}\cL_{\Bar{b}}\ot {}_{\Bar{c},d}\cL_{\Bar{a},b}$.
Let $X\ot Y\in \cL\ot_{\BB}\cL$ and $Z\ot W\ot V\in \int^{a,b}\int_{c,d}{}_{\Bar{a}}\cL_{\Bar{c}}\ot {}_{\Bar{d}}\cL_{\Bar{b}}\ot {}_{\Bar{c},d}\cL_{\Bar{a},b}$. We can see
\begin{align*}
    X\Bar{b}Z\ot \sigma(V, Y)(W\la n)m= XZ\ot \sigma(V\Bar{b}, Y)(W\la n)m=XZ\ot \sigma(V, \Bar{b}Y)(W\la n)m,
\end{align*}
and
\begin{align*}
    XZ\Bar{b}\ot_{\BB} \sigma(V, Y)(W\la n)m=XZ\ot_{\BB} \sigma(\Bar{b}V, Y)(W\la n)m,
\end{align*}
    and
    \begin{align*}
    XZ\ot \sigma(V, Y)(\Bar{b}W\la n)m=XZ\ot \sigma(V, Y)\Bar{b}(W\la n)m=XZ\ot \sigma(bV, Y)(W\la n)m,
\end{align*}
so the product is well defined. Now, let's check the product is associative. On the one hand, for any $X\ot n, Y\ot m, Z\ot l\in \cL\ot_{\BB}N$
\begin{align*}
    ((X\ot n)&(Y\ot m))(Z\ot l)\\
    =&X_{+}Y_{+}Z_{+}\ot \sigma(Z_{-}\t, Y_{+-}X_{+-})(Z_{-}\o\la (\sigma(Y_{-}\t, X_{-})(Y_{-}\o\la n)m))l\\
     =&X_{+}Y_{+}Z_{+}\ot \sigma(Z_{-}\t, Y_{-}\th X_{-}\t)(Z_{-}\o\la (\sigma(Y_{-}\t, X_{-}\o)(Y_{-}\o\la n)m))l\\
     =&X_{+}Y_{+}Z_{+}\ot \sigma(Z_{-}\th, Y_{-}\th X_{-}\t)(Z_{-}\t\la \sigma(Y_{-}\t, X_{-}\o))(Z_{-}\o\la((Y_{-}\o\la n)m))l,
\end{align*}
on the other hand,
\begin{align*}
    (X\ot n)&((Y\ot m)(Z\ot l))\\
    =&X_{+}Y_{+}Z_{+}\ot \sigma(Z_{+-}\t Y_{+-}\t, X_{-})((Z_{+-}\o Y_{+-}\o)\la n)\sigma(Z_{-}\t, Y_{-})(Z_{-}\o\la m)l\\
    =&X_{+}Y_{+}Z_{+}\ot \sigma(Z_{-}\fo Y_{-}\th, X_{-})((Z_{-}\th Y_{-}\t)\la n)\sigma(Z_{-}\t, Y_{-}\o)(Z_{-}\o\la m)l\\
    =&X_{+}Y_{+}Z_{+}\ot \sigma(Z_{-}\fo Y_{-}\th, X_{-})\sigma(Z_{-}\th, Y_{-}\t)((Z_{-}\t\la (Y_{-}\o\la n))(Z_{-}\o\la m)l,
\end{align*}
where the 3rd step uses (v). So the product is associative by using (iv).

Conversely, if $\cL$ is a Hopf algebroid over $B$, $N$ is a $\BB$-ring and $(\sigma,\la)$ satisfy all the conditions as above without (iv) and (v), then the product (\ref{equ. product of the twisted crossed product}) is associative implies $(\sigma,\la)$ satisfies condition (iv) and (v). We can see (iv) can be proved by comparing $Z_{[+]}Y_{[+]}Z_{[+]}((Z_{[-]}\ot 1)(Y_{[-]}\ot 1))(X_{[-]}\ot 1)$ and $Z_{[+]}Y_{[+]}Z_{[+]}(Z_{[-]}\ot 1)((Y_{[-]}\ot 1)(X_{[-]}\ot 1))$. Indeed, it is not hard to see both of them are well defined, and
    \begin{align*}
        Z_{[+]}&Y_{[+]}X_{[+]}((Z_{[-]}\ot 1)(Y_{[-]}\ot 1))(X_{[-]}\ot 1)\\
        =&Z_{[+]}Y_{[+]}X_{[+]}X_{[-]+}Y_{[-]+}Z_{[-]+}\ot \sigma(Z_{[-]-}\t, Y_{[-]-}\t X_{[-]-}\t)(Z_{[-]-}\o\la \sigma(Y_{[-]-}\o, X_{[-]-}\o))\\
        =&Z\t{}_{[+]}Y\t{}_{[+]}X\t{}_{[+]}X\t{}_{[-]}Y\t{}_{[-]}Z\t{}_{[-]} \ot \sigma(Z\o\t, Y\o\t X\o\t)(Z\o\o\la \sigma(Y\o\o, X\o\o))\\
        =&\overline{\varepsilon(Z\t Y\t X\t)}\ot \sigma(Z\o\t, Y\o\t X\o\t)(Z\o\o\la \sigma(Y\o\o, X\o\o))\\
        =&1\ot\sigma(Z\t, Y\t X\t)Z\o\la\sigma(Y\o, X\o)
    \end{align*}
    where the 2nd step uses \ref{prop. anti left and left Galois maps}. Similarly,
    \begin{align*}
        Z_{[+]}&Y_{[+]}Z_{[+]}(Z_{[-]}\ot 1)((Y_{[-]}\ot 1)(X_{[-]}\ot 1))
        =1\ot \sigma(Z\t Y\t, X)\sigma(Z\o, Y\o).
    \end{align*}
    Similarly, (v) can be proved by comparing $Y_{[+]}X_{[+]}((1\ot n)(X_{[-]}\ot 1))(Y_{[-]}\ot 1)$ and $Y_{[+]}X_{[+]}(1\ot n)((X_{[-]}\ot 1)(Y_{[-]}\ot 1))$. Indeed,
    \begin{align*}
        Y_{[+]}&X_{[+]}((1\ot n)(X_{[-]}\ot 1))(Y_{[-]}\ot 1)\\
        =&Y_{[+]}X_{[+]}X_{[-]++}Y_{[-]+}\ot \sigma(Y_{[-]-}\t\ot, X_{[-]+-})Y_{[-]-}\o\la (X_{[-]-}\la n)\\
        =&Y_{[+]}X_{[+]}X_{[-]+}Y_{[-]+}\ot \sigma(Y_{[-]-}\t, X_{[-]-}\t)Y_{[-]-}\o\la (X_{[-]-}\o\la n)\\
        =&1\ot \sigma(Y\t, X\t)Y\o\la(X\o\la n),
    \end{align*}
    and similarly,
    \begin{align*}
        Y_{[+]}X_{[+]}(1\ot n)((X_{[-]}\ot 1)(Y_{[-]}\ot 1))=1\ot ((Y\t X\t)\la n) \sigma(Y\o, X\o).
    \end{align*}
\end{proof}

\begin{lem}\label{lem. twisted crossed product is cleft extension}
        Let $\cL$ be a Hopf algebroid over $B$ and $N$ be a $\BB$-ring. If $(\sigma, \la)$ is a 2-cocycle of $\cL^{cop}$ over $N$ and $N$ is a $\sigma$-twisted left $\cL^{cop}$-module algebra, then $N\subseteq\cL\#_{\sigma}N$ is a cleft extension with cleaving map $\gamma:\cL\to \cL\#_{\sigma}N$ given by $\gamma(X):=X\#_{\sigma} 1$, and
        \begin{align}\label{equ. inverse of cocycle cleaving map}
            X^{\alpha}\ot_{\BB} X^{\beta}=X_{+}\ot_{\BB} X_{-}\o{}_{+}\#_{\sigma}\sigma^{-1}(X_{-}\o{}_{-}, X_{-}\t),
        \end{align}
        for any $X\in \cL$.
\end{lem}

\begin{proof}
We can see $\cL\#_{\sigma}N$ is a $\BB$-bimodule with the $\BB$-bimodule structure given by $\Bar{a}(X\ot n)\Bar{b}=\Bar{a}X\ot n\Bar{b}$ for any $\Bar{a}, \Bar{b}\in \BB$ and $X\ot n\in \cL\#_{\sigma}N$.
$\cL\#_{\sigma}N$ is also a left $\cL$-comodule with the left coaction on $\cL$, for which the $B$-bimodule structure is given by $a(X\ot n)b=aXb\ot n$ for any $a,b\in B$. So it is not hard to see $\gamma$ is left $\cL$-colinear and $\BB$-bilinear. Now, we check (\ref{equ. inverse of cocycle cleaving map}) is well defined.
We can see for any $X\in\cL$
\begin{align*}
    X_{+}\ot X_{-}\o{}_{+}\ot X_{-}\o{}_{-} \ot X_{-}\t\in \int^{a,b,c}\int_{d,e,f}{}_{\Bar{a}}\cL_{\Bar{d}}\ot {}_{\Bar
    {b}}\cL_{\Bar{e}}\ot {}_{c,\Bar{e}}\cL_{\Bar{b},f}\ot {}_{\Bar{d},f}\cL_{\Bar{a},c}
\end{align*}
Let $X\ot Y\ot Z\ot W\in \int^{a,b,c}\int_{d,e,f}{}_{\Bar{a}}\cL_{\Bar{d}}\ot {}_{\Bar
    {b}}\cL_{\Bar{e}}\ot {}_{c,\Bar{e}}\cL_{\Bar{b},f}\ot {}_{\Bar{d},f}\cL_{\Bar{a},c}$ and $b\in B$, we have
    \begin{align*}
         X\Bar{b}\ot_{\BB} Y\#_{\sigma}\sigma^{-1}(Z, W)=& X\ot_{\BB} \Bar{b}Y\#_{\sigma}\sigma^{-1}(Z, W)
         =X\ot_{\BB} Y\#_{\sigma}\sigma^{-1}(Z\Bar{b}, W)\\
         =&X\ot_{\BB} Y\#_{\sigma}\sigma^{-1}(Z, \Bar{b}W),
    \end{align*}
and
\begin{align*}
          X\ot_{\BB} Y\Bar{b}\#_{\sigma}\sigma^{-1}(Z, W)
         =X\ot_{\BB} Y\#_{\sigma}\Bar{b}\sigma^{-1}(Z, W)
         =X\ot_{\BB} Y\#_{\sigma}\sigma^{-1}(\Bar{b}Z, W),
    \end{align*}
and
\begin{align*}
         X\ot_{\BB} Y\#_{\sigma}\sigma^{-1}(Zb, W)= X\ot_{\BB} Y\#_{\sigma}\sigma^{-1}(Z, bW),
    \end{align*}
so (\ref{equ. inverse of cocycle cleaving map}) is well defined. One the one hand,
\begin{align*}
    X^{\alpha}\o&\ot \gamma(X^{\alpha}\t)X^{\beta}\\
    =&X_{+}\o\ot (X_{+}\t\#_{\sigma}1)(X_{-}\o{}_{+}\#_{\sigma}\sigma^{-1}(X_{-}\o{}_{-}, X_{-}\t))\\
    =&X_{+}\o\ot X_{+}\t{}_{+}X_{-}\o{}_{++}\#_{\sigma}\sigma(X_{-}\o{}_{+-}, X_{+}\t{}_{-})\sigma^{-1}(X_{-}\o{}_{-}, X_{-}\t)\\
    =&X_{++}\o\ot X_{++}\t X_{-}\o{}_{+}\#_{\sigma}\sigma(X_{-}\o{}_{-}\t, X_{+-})\sigma^{-1}(X_{-}\o{}_{-}\o, X_{-}\t)\\
    =&X_{+}\o\ot X_{+}\t X_{-}\o{}_{+}\#_{\sigma}\sigma(X_{-}\o{}_{-}\t, X_{-}\th)\sigma^{-1}(X_{-}\o{}_{-}\o, X_{-}\t)\\
    =&X_{+}\o\ot X_{+}\t X_{-}\o{}_{+}\#_{\sigma}\overline{\varepsilon(X_{-}\o{}_{-} X_{-}\t)}\\
    =&X_{+}\o\ot X_{+}\t X_{-}\o{}_{+}\#_{\sigma}\overline{\varepsilon(X_{-}\o{}_{-} \varepsilon(X_{-}\t))}\\
    =&X_{+}\o\ot X_{+}\t (\overline{(\varepsilon(X_{-}\t)}X_{-}\o){}_{+}\#_{\sigma}\overline{\varepsilon(\overline{(\varepsilon(X_{-}\t)}X_{-}\o){}_{-})}\\
    =&X_{+}\o\ot X_{+}\t X_{-+}\#_{\sigma}\overline{\varepsilon(X_{--})}\\
    =&X_{+}\o\ot X_{+}\t X_{-+}\overline{\varepsilon(X_{--})}\#_{\sigma}1\\
    =&X_{+}\o\ot X_{+}\t X_{-}\#_{\sigma}1\\
    =&X\ot 1\#_{\sigma}1.
\end{align*}
On the other hand,
\begin{align*}
X\o&{}^{\alpha}\ot X\o{}^{\beta}\gamma(X\t)\\
=&X\o{}_{+}\ot_{\BB} (X\o{}_{-}\o{}_{+}\#_{\sigma}\sigma^{-1}(X\o{}_{-}\o{}_{-}, X\o{}_{-}\t)(X\t\#_{\sigma}1)\\
=&X\o{}_{+}\ot_{\BB} X\o{}_{-}\o{}_{++}X\t{}_{+}\#_{\sigma}\sigma(X\t{}_{-}\t, X\o{}_{-}\o{}_{+-})X\t{}_{-}\o\la \sigma^{-1}(X\o{}_{-}\o{}_{-},X\o{}_{-}\t)\\
=&X\o{}_{+}\ot_{\BB} X\o{}_{-}\o{}_{+}X\t{}_{+}\#_{\sigma}\sigma(X\t{}_{-}\t, X\o{}_{-}\o{}_{-}\t)X\t{}_{-}\o\la \sigma^{-1}(X\o{}_{-}\o{}_{-}\o,X\o{}_{-}\t)\\
=&X\o{}_{++}\ot_{\BB} X\o{}_{-+}X\t{}_{+}\#_{\sigma}\sigma(X\t{}_{-}\t, X\o{}_{--}\t)X\t{}_{-}\o\la \sigma^{-1}(X\o{}_{--}\o,X\o{}_{+-})\\
=&X_{++}\ot_{\BB} X_{-[+]+}X_{-[-]+}\#_{\sigma}\sigma(X_{-[-]-}\t, X_{-[+]-}\t)X_{-[-]-}\o\la \sigma^{-1}(X_{-[+]-}\o,X_{+-})\\
=&X_{++}\ot_{\BB} X_{-+[+]}X_{-+[-]+}\#_{\sigma}\sigma(X_{-+[-]-}\t, X_{--}\t)X_{-+[-]-}\o\la \sigma^{-1}(X_{--}\o,X_{+-})\\
=&X_{++}\ot_{\BB} X_{-+}\t{}_{[+]}X_{-+}\t{}_{[-]}\#_{\sigma}\sigma(X_{-+}\o\t, X_{--}\t)X_{-+}\o\o\la \sigma^{-1}(X_{--}\o,X_{+-})\\
=&X_{++}\ot_{\BB} \overline{\varepsilon(X_{-+}\t)}\#_{\sigma}\sigma(X_{-+}\o\t, X_{--}\t)X_{-+}\o\o\la \sigma^{-1}(X_{--}\o,X_{+-})\\
=&X_{++}\ot_{\BB} 1\#_{\sigma}\sigma(X_{-+}\t, X_{--}\t)X_{-+}\o\la \sigma^{-1}(X_{--}\o,X_{+-})\\
=&X_{+}\ot_{\BB} 1\#_{\sigma}\sigma(X_{-}\o{}_{+}\t, X_{-}\o{}_{-}\t)X_{-}\o{}_{+}\o\la \sigma^{-1}(X_{-}\o{}_{-}\o,X_{-}\t)\\
=&X_{+}\ot_{\BB} 1\#_{\sigma}\sigma(X_{-}\t{}_{+}, X_{-}\t{}_{-}\t)X_{-}\o\la \sigma^{-1}(X_{-}\t{}_{-}\o,X_{-}\th)\\
=&X_{+}\ot_{\BB} 1\#_{\sigma}\overline{\varepsilon(X_{-})}\\
=&X\ot_{\BB} 1\#_{\sigma} 1,
\end{align*}
where the 12th step uses Proposition \ref{prop. proposition of 2-cocycle over a ring}.
\end{proof}

Recall that, given a Hopf algebra, and a left cleft extension $N\subseteq P$ with cleaving map $\gamma$, one can construct a `right-handed' measuring $\ra$ and 2-cocycle $\Tilde{\sigma}$ by
\[n\ra h:=\gamma(h\o)n\gamma^{-1}(h\t),\quad \Tilde{\sigma}(h,g):=\gamma^{-1}(h\o,g\o)\gamma(h\t)\gamma(g\t).\]
However, for a $\cL$-cleft extension $N\subseteq P$, as we do not have the inverse of the cleaving map, we cannot define a measuring and 2-cocycle in the same way. However, we can still define a `left handed' measuring and 2-cocycle fitting the idea outlined in Remark \ref{rem. classical explanation for cleft extension}, as proved in the following result.

\begin{thm}\label{thm. equivalence between cleft extension and twist crossed product}
    Let $\cL$ be a Hopf algebroid over $B$ and $P$ be a left $\cL$-comodule algebra with $N={}^{co\cL}P$, then $N\subseteq P$ is a cleft $\cL$-extension if and only if $P\simeq \cL\#_{\sigma}N$ as comodule algebras.
\end{thm}
\begin{proof}
    By the Lemma above, if $P$ is isomorphic to $\cL\#_{\sigma}N$, then $N\subseteq P$ is a cleft $\cL$-extension. Conversely, assume $N\subseteq P$ is a cleft $\cL$-extension with cleaving map $\gamma:\cL\to P$. We can construct a measuring $\la:\cL^{cop}\ot_{\BB}N\to N$ by
    \begin{align}\label{equ. measuring from a cleft extension}
        X\la n:=X_{[+]}X_{[-]}\o{}^{\alpha}\ot_{\BB}X_{[-]}\o{}^{\beta}n\gamma(X_{[-]}\t)\in \BB\ot_{\BB}N\simeq N,
    \end{align}
    for any $X\in \cL$ and $n\in N$.
    We will first check (\ref{equ. measuring from a cleft extension}) is well defined. We can see
    \begin{align*}
        X_{[+]}\ot X_{[-]}\o{}^{\alpha}\ot X_{[-]}\o{}^{\beta}\ot X_{[-]}\t\in \int^{a,b,c}\int_{d,e,f}{}_{a}\cL_{d}\ot {}_{\Bar{b},d}\cL_{a,\Bar{e}}\ot {}_{c,\Bar{e}}\cL_{\Bar{b},f}\ot {}_{f}\cL_{c}.
    \end{align*}
    Let $X\ot Y\ot Z\ot W\in \int^{a,b,c}\int_{d,e,f}{}_{a}\cL_{d}\ot {}_{\Bar{b},d}\cL_{a,\Bar{e}}\ot {}_{c,\Bar{e}}\cL_{\Bar{b},f}\ot {}_{f}\cL_{c}$, we can see
    \begin{align*}
        XY\ot_{\BB}Zb n\gamma(W)=XY\ot_{\BB}Z nb\gamma(W)=XY\ot_{\BB}Z n\gamma(bW),
    \end{align*}
    where the 1st step use the fact that $N$ commutes with $B$ as subalgebras in $P$ and the 2nd step uses the fact that $\gamma$ is $B$-bilinear. So (\ref{equ. measuring from a cleft extension}) is well defined. To see $X_{[+]}X_{[-]}\o{}^{\alpha}\ot_{\BB}X_{[-]}\o{}^{\beta}n\gamma(X_{[-]}\t)\in \BB\ot_{\BB}N$, we can check
    \begin{align*}
        (\Delta\ot \id_{N})&(X_{[+]}X_{[-]}\o{}^{\alpha}\ot_{\BB}X_{[-]}\o{}^{\beta}n\gamma(X_{[-]}\t))\\
        =&X_{[+]}\o X_{[-]}\o{}^{\alpha}\o\ot X_{[+]}\t X_{[-]}\t{}^{\alpha}\t\ot X_{[-]}\o{}^{\beta}n\gamma(X_{[-]}\t)\\
        =&X_{[+]}\o X_{[-]}\o\ot X_{[+]}\t X_{[-]}\t{}^{\alpha}\ot X_{[-]}\t{}^{\beta}n\gamma(X_{[-]}\th)\\
        =&X_{[+][+]}\o X_{[+][-]}\ot X_{[+][+]}\t X_{[-]}\o{}^{\alpha}\ot X_{[-]}\o{}^{\beta}n\gamma(X_{[-]}\t)\\
        =&1\ot X_{[+]}X_{[-]}\o{}^{\alpha}\ot X_{[-]}\o{}^{\beta}n\gamma(X_{[-]}\t),
    \end{align*}
    where the 2nd step uses (\ref{equ. cleaving map 3}). We also have
    \begin{align*}
        (\id_{\cL}\ot \delta)&(X_{[+]}X_{[-]}\o{}^{\alpha}\ot_{\BB}X_{[-]}\o{}^{\beta}n\gamma(X_{[-]}\t))\\
        =&X_{[+]}X_{[-]}\o{}^{\alpha}\ot X_{[-]}\o{}^{\beta}\mo\gamma(X_{[-]}\t)\mo\ot X_{[-]}\o{}^{\beta}\z n\gamma(X_{[-]}\t)\z\\
        =&X_{[+]}X_{[-]}\o{}_{+}{}^{\alpha}\ot X_{[-]}\o{}_{-} X_{[-]}\t \ot X_{[-]}\o{}_{+}{}^{\beta} n\gamma(X_{[-]}\th)\\
        =&X_{[+]}X_{[-]}\o{}^{\alpha}\ot 1\ot X_{[-]}\o{}^{\beta}n\gamma(X_{[-]}\t),
    \end{align*}
    where the 2nd step uses  (\ref{equ. cleaving map 4}). Since $X_{[+]}X_{[-]}\o{}^{\alpha}\ot_{\BB}X_{[-]}\o{}^{\beta}n\gamma(X_{[-]}\t)\in \BB\ot_{\BB}N$, we can write $X\la n=(X_{[+]}X_{[-]}\o{}^{\alpha})X_{[-]}\o{}^{\beta}n\gamma(X_{[-]}\t)$. We check that $\la$ is a measuring, we can see that
\begin{align*}
    (b\Bar{b'}X)\la n=(\Bar{b'}X_{[+]}X_{[-]}\o{}^{\alpha})X_{[-]}\o{}^{\beta}n\gamma(X_{[-]}\t \Bar{b})=\Bar{b'}(X\la n)\Bar{b}
\end{align*}
    for any $b,b'\in B$, $n\in N$ and $X\in \cL$. And
    \begin{align*}
        X\la \Bar{b}=&(X_{[+]}X_{[-]}\o{}^{\alpha})X_{[-]}\o{}^{\beta}\Bar{b}\gamma(X_{[-]}\t)=(X_{[+]}\Bar{b}X_{[-]}\o{}^{\alpha})X_{[-]}\o{}^{\beta}\gamma(X_{[-]}\t)\\
        =&((X\Bar{b})_{[+]}(X\Bar{b})_{[-]}\o{}^{\alpha})(X\Bar{b})_{[-]}\o{}^{\beta}\gamma((X\Bar{b})_{[-]}\t)=(X\Bar{b})_{[+]}(X\Bar{b})_{[-]}\\
        =&\overline{\varepsilon(X\Bar{b})},
    \end{align*}
    where the 2nd step uses (\ref{equ. cleaving map 6}). Moreover, for any $n, m\in N$ we have
    \begin{align*}
    (&X\t\la n)(X\o\la m)\\
=&(X\t{}_{[+]}X\t{}_{[-]}\o{}^{\alpha})X\t{}_{[-]}\o{}^{\beta}n\gamma(X\t{}_{[-]}\t)(X\o{}_{[+]}X\o{}_{[-]}\o{}^{\alpha})X\o{}_{[-]}\o{}^{\beta}m\gamma(X\o{}_{[-]}\t)\\
        =&(X\t{}_{[+]}X\t{}_{[-]}\o{}^{\alpha})X\t{}_{[-]}\o{}^{\beta}n\gamma(X\t{}_{[-]}\t X\o{}_{[+]}X\o{}_{[-]}\o{}^{\alpha})X\o{}_{[-]}\o{}^{\beta}m\gamma(X\o{}_{[-]}\t)\\
        =&(X_{[+]}\t{}_{[+]}X_{[+]}\t{}_{[-]}\o{}^{\alpha})X_{[+]}\t{}_{[-]}\o{}^{\beta}n\gamma(X_{[+]}\t{}_{[-]}\t X{}_{[+]}\o X{}_{[-]}\o{}^{\alpha})X{}_{[-]}\o{}^{\beta}m\gamma(X{}_{[-]}\t)\\
         =&(X_{[+]}\t{}_{[+][+]}X_{[+]}\t{}_{[+][-]}{}^{\alpha})X_{[+]}\t{}_{[+][-]}{}^{\beta}n\gamma(X_{[+]}\t{}_{[-]} X{}_{[+]}\o X{}_{[-]}\o{}^{\alpha})X{}_{[-]}\o{}^{\beta}m\gamma(X{}_{[-]}\t)\\
         =&(X_{[+][+]}X_{[+][-]}{}^{\alpha})X_{[+][-]}{}^{\beta}n\gamma(X{}_{[-]}\o{}^{\alpha})X{}_{[-]}\o{}^{\beta}m\gamma(X{}_{[-]}\t)\\
         =&(X_{[+][+]}X_{[+][-]}{}^{\alpha})X_{[+][-]}{}^{\beta}nm\varepsilon(X_{[-]}\o)\gamma(X{}_{[-]}\t)\\
          =&(X_{[+][+]}X_{[+][-]}{}^{\alpha})X_{[+][-]}{}^{\beta}nm\gamma(X{}_{[-]})\\
          =&(X_{[+]}X_{[-]}\o{}^{\alpha})X_{[-]}\o{}^{\beta}nm\gamma(X_{[-]}\t)\\
          =&X\la(nm),
    \end{align*}
    where the 6th step uses the fact that $\gamma(X^{\alpha})X^{\beta}=\varepsilon(X)\in B\subseteq P$ and $B$ commute with $N$ as subalgebras in $P$. In the following, we will always use the above method, for convenience we call it the method of cancelation.

    Now, we can construct a 2-cocycle on of $\cL^{cop}$ over $N$ by
    \begin{align}\label{equ. 2-cocycle from cleft extension}
        \sigma(X, Y):=X_{[+]}Y_{[+]}(Y_{[-]}\o X_{[-]}\o)^{\alpha}\ot_{\BB}(Y_{[-]}\o X_{[-]}\o)^{\beta}\gamma(Y_{[-]}\t)\gamma(X_{[-]}\t)\in \BB\ot_{\BB}N,
    \end{align}
    by almost the same method we can see $\sigma:\cL^{cop}\ot_{B^{e}}\cL^{cop}\to N$ is well defined and its image belongs to $\BB\ot_{\BB}N\simeq N$. Here we first check $\sigma$ factors through the balanced tensor product $\ot_{B^{e}}$. Indeed, for any $b, b'\in B$ and $X,Y\in \cL$
    \begin{align*}
        \sigma(Xb\Bar{b'}, Y)=&X_{[+]}\Bar{b'}Y_{[+]}(Y_{[-]}\o X_{[-]}\o)^{\alpha}\ot_{\BB}(Y_{[-]}\o X_{[-]}\o)^{\beta}\gamma(Y_{[-]}\t)\gamma(\Bar{b}X_{[-]}\t)\\
        =&X_{[+]}\Bar{b'}Y_{[+]}(Y_{[-]}\o X_{[-]}\o)^{\alpha}\ot_{\BB}(Y_{[-]}\o X_{[-]}\o)^{\beta}\gamma(Y_{[-]}\t)\Bar{b}\gamma(X_{[-]}\t)\\
        =&X_{[+]}\Bar{b'}Y_{[+]}(Y_{[-]}\o X_{[-]}\o)^{\alpha}\ot_{\BB}(Y_{[-]}\o X_{[-]}\o)^{\beta}\gamma(Y_{[-]}\t\Bar{b})\gamma(X_{[-]}\t)\\
        =&\sigma(X, b\Bar{b'}Y)
    \end{align*}
    where the 2nd step use $\gamma$ is left $\BB$-linear and the 3rd step uses $\gamma$ is right $\BB$-linear. By the same method, we can also see
    \begin{align*}
        \sigma(b\Bar{b'}X, Y)=\Bar{b'}\sigma(X, Y)\Bar{b},
    \end{align*}
    and
    \begin{align*}
        \sigma(X, 1)=\sigma(1, X)=\overline{\varepsilon(X)},
    \end{align*}
and
\begin{align*}
    \sigma(X, Yb)=&X_{[+]}Y_{[+]}(Y_{[-]}\o X_{[-]}\o)^{\alpha}\ot_{\BB}(Y_{[-]}\o X_{[-]}\o)^{\beta}\gamma(\Bar{b}Y_{[-]}\t)\gamma(X_{[-]}\t)\\
    =&X_{[+]}Y_{[+]}(Y_{[-]}\o X_{[-]}\o)^{\alpha}\ot_{\BB}(Y_{[-]}\o X_{[-]}\o)^{\beta}\Bar{b}\gamma(Y_{[-]}\t)\gamma(X_{[-]}\t)\\
    =&X_{[+]}Y_{[+]}\Bar{b}(Y_{[-]}\o X_{[-]}\o)^{\alpha}\ot_{\BB}(Y_{[-]}\o X_{[-]}\o)^{\beta}\gamma(Y_{[-]}\t)\gamma(X_{[-]}\t)\\
    =&\sigma(X, Y\Bar{b}),
\end{align*}
    where the 3rd step uses (\ref{equ. cleaving map 6}).  By Lemma \ref{lem. twisted module algebra and crossed product}, in order to show $\sigma$ satisfies the 2-cocycle condition (iv) and (v), it is sufficient to show the twisted crossed product $\cL\#_{\sigma}N$ is associative. We can see
    \begin{align*}
        (&X\#_{\sigma}n)(Y\#_{\sigma}m)\\
        =&X_{+}Y_{+}\#_{\sigma} \sigma(Y_{-}\t, X_{-})(Y_{-}\o\la n)m\\
        =&X_{+}Y_{+}\#_{\sigma} \sigma(Y_{-}\t, X_{-})(Y_{-}\o\la n)m\\
        =&X_{+}Y_{+}\#_{\sigma}(Y_{-}\t{}_{[+]} X_{-[+]}(X_{-[-]}\o Y_{-}\t{}_{[-]}\o)^{\alpha})(X_{-[-]}\o Y_{-}\t{}_{[-]}\o)^{\beta}\gamma(X_{-[-]}\t)\\
        &\gamma(Y_{-}\t{}_{[-]}\t)(Y_{-}\o{}_{[+]} Y_{-}\o{}_{[-]}\o{}^{\alpha})Y_{-}\o{}_{[-]}\o{}^{\beta}n\gamma(Y_{-}\o{}_{[-]}\t)m\\
        =&X_{+}Y_{+}(Y_{-}\t{}_{[+]} X_{-[+]}(X_{-[-]}\o Y_{-}\t{}_{[-]}\o)^{\alpha})\#_{\sigma}(X_{-[-]}\o Y_{-}\t{}_{[-]}\o)^{\beta}\gamma(X_{-[-]}\t)\\
        &\gamma(Y_{-}\t{}_{[-]}\t)(Y_{-}\o{}_{[+]} Y_{-}\o{}_{[-]}\o{}^{\alpha})Y_{-}\o{}_{[-]}\o{}^{\beta}n\gamma(Y_{-}\o{}_{[-]}\t)m\\
         =&X{}_{+}Y{}_{+}(Y{}_{-[+]} X{}_{-[+]}(X_{-[-]}\o Y_{-[-]}\o)^{\alpha})\#_{\sigma}(X_{-[-]}\o Y_{-[-]}\o)^{\beta}\gamma(X_{-[-]}\t)n\gamma(Y_{-[-]}\t)m\\
        =&X\o{}_{+}Y\o{}_{+}(Y\o{}_{-} X\o{}_{-}(X\t Y\t)^{\alpha})\#_{\sigma}(X\t Y\t)^{\beta}\gamma(X\th)n\gamma(Y\th)m\\
        =&(X\o Y\o)^{\alpha}\#_{\sigma} (X\o Y\o)^{\beta}\gamma(X\t)n\gamma(Y\t)m
    \end{align*}
    where the 5th step uses the method of cancelation as above. By Proposition \ref{prop. algebra structure given by cleft extension}, $\cL\#_{\sigma}N$ is equal to $\cL\#^{\gamma}N$ as algebra. Therefore, by Lemma \ref{lem. twisted module algebra and crossed product}, $(\sigma, \la)$ is a 2-cocycle. and $N$ is a $\sigma$-twisted left $\cL^{cop}$-module algebra. Moreover, by Lemma \ref{lem. cleft extension has normal basis property}, $P\simeq \cL\#^{\gamma} N = \cL\#_{\sigma} N$ as comodule algebras.

\end{proof}

\subsection{Equivalence classes of cleft extensions}

\begin{defi}
    Let $\cL$ be a left $B$-bialgebroid and $M$ is a $\BB$-ring. We define $C^{1}(\cL, M)$ to be the set of convolution invertible maps $u:\cL\to M$ such that
    \begin{align}\label{eq1}
 u(b\Bar{b'}X)=\Bar{b'}u(X)\Bar{b},\quad  u(Xb)=u(X\Bar{b}).
\end{align}

\end{defi}

This definition is a generalization of the one given in \cite{HM23}, which is the case for $M=B$.

\begin{prop}
        Let $(\sigma,\la)$ be a 2-cocycle which makes $N$ a $\sigma$-twisted left $\cL^{cop}$-module algebra. If $u\in C^{1}(\cL^{cop}, N)$, then the pair of maps $(\sigma',\la')$ given by
        \begin{align}
            X\la' n=&u(X\th)(X\t\la n)u^{-1}(X\o)\\
            \sigma'(X, Y)=&u(X\fo Y\th)\sigma(X\th, Y\t)(X\t\la u^{-1}(Y\o))u^{-1}(X\o),
        \end{align}
        for any $X, Y\in \cL$ and $n\in N$ is a 2-cocycle and makes $N$ a $\sigma'$-twisted left $\cL^{cop}$-module algebra. We called $(\sigma,\la)$ and $(\sigma',\la')$ are equivalent (via $u\in C^{1}(\cL^{cop}, N)$) if they satisfy the above relation.
\end{prop}

We will omit the proof as it is a direct computation by checking $(\sigma',\la')$ to satisfy all the axioms of a 2-cocycle. Under the same idea of Remark \ref{rem. classical explanation for cleft extension}, we have

\begin{thm}
     Let $\cL$ be a left $B$-Hopf algebroid, $N$ be a $\BB$-ring and $(\sigma,\la)$,  $(\sigma',\la')$ be a pair of 2-cocycles that make $N$  twisted left $\cL^{cop}$-module algebras respectively, then $(\sigma,\la)$ and $(\sigma',\la')$ are equivalent if and only if $\cL\#_{\sigma}N\simeq \cL\#_{\sigma'}N$ as left $\cL$-comodule algebra.
\end{thm}
\begin{proof}
    Assume $(\sigma,\la)$,  $(\sigma',\la')$ are equivelant via $u\in C^{1}(\cL^{cop}, N)$,  we can define an isomorphism $\Phi:\cL\#_{\sigma}N\to \cL\#_{\sigma'}N$ by
    \begin{align}
        \Phi(X\#_{\sigma} n):=X_{+}\#_{\sigma'} u(X_{-})n,
    \end{align}
    for any $X\#_{\sigma} n\in \cL\#_{\sigma}N$. It is not hard to see $\Phi$ is well defined. We can also see
    \begin{align*}
        X_{+}\o\ot X_{+}\t\#_{\sigma'} u(X_{-})n=X\o\ot X\t{}_{+}\#_{\sigma'} u(X\t{}_{-})n,
    \end{align*}
    so $\Phi$ is left $\cL$-colinear. Moreover, $\Phi$ is an algebra map. Indeed, on the one hand
    \begin{align*}
        \Phi((X\#_{\sigma}n)(Y\#_{\sigma}m))=&X_{++}Y_{++}\#_{\sigma'}u(Y_{+-}X_{+-})\sigma(Y_{-}\t, X_{-})(Y_{-}\o\la n)m\\
        =&X_{+}Y_{+}\#_{\sigma'}u(Y_{-}\th X_{-}\t)\sigma(Y_{-}\t, X_{-}\o)(Y_{-}\o\la n)m,
    \end{align*}
    on the other hand
    \begin{align*}
        \Phi(X\#_{\sigma}n)\Phi(Y\#_{\sigma}m)=&X_{++}Y_{++}\#_{\sigma'}\sigma'(Y_{+-}\t, X_{+-})(Y_{+-}\t\la'(u(X_{-})n))u(Y_{-})m\\
        =&X_{+}Y_{+}\#_{\sigma'}\sigma'(Y_{-}\th, X_{-}\t)(Y_{-}\t\la'(u(X_{-}\o)n))u(Y_{-}\o)m\\
        =&X_{+}Y_{+}\#_{\sigma'}u(Y_{-}\th X_{-}\t)\sigma(Y_{-}\t, X_{-}\o)(Y_{-}\o\la n)m.
    \end{align*}

    Conversely, let $\cL\#_{\sigma}N\simeq \cL\#_{\sigma'}N$ and $\gamma, \gamma'$ be the corresponding cleaving map of $(\sigma,\la)$ and $(\sigma',\la')$ respectively by Lemma \ref{lem. twisted crossed product is cleft extension}. We can define $u\in C^{1}(\cL^{cop}, N)$ by
    \begin{align}
        u(X):=X_{[+]}X_{[-]}\o{}^{\alpha'}\ot X_{[-]}\o{}^{\beta'} \gamma(X_{[-]}\t)\in \BB\ot_{\BB}N\simeq N
    \end{align}
    where $X^{\alpha'}\ot X^{\beta'}$ is the image of the corresponding map of $\gamma'$. By almost the same reason as (\ref{equ. measuring from a cleft extension}), we can see $u$ is well defined and satisfies $u(b\Bar{b'}X)=\Bar{b'}u(X)\Bar{b}$. Also,
    \begin{align*}
        u(Xb)=&X_{[+]}X_{[-]}\o{}^{\alpha'}\ot X_{[-]}\o{}^{\beta'}\gamma(\Bar{b}X_{[-]}\t)
        =X_{[+]}X_{[-]}\o{}^{\alpha'}\ot X_{[-]}\o{}^{\beta'}\Bar{b}\gamma(X_{[-]}\t)\\
        =&X_{[+]}\Bar{b}X_{[-]}\o{}^{\alpha'}\ot X_{[-]}\o{}^{\beta'}\gamma(X_{[-]}\t)
        =u(X\Bar{b}),
    \end{align*}
    where the 3rd step uses (\ref{equ. cleaving map 6}). Moreover, it is convolution invertible with
    \begin{align}
        u^{-1}(X)=X_{[+]}X_{[-]}\o{}^{\alpha}\ot X_{[-]}\o{}^{\beta} \gamma'(X_{[-]}\t).
    \end{align}
    Indeed,
    \begin{align*}
        u^{-1}&(X\t)u(X\o)\\
        =&(X\t{}_{[+]}X\t{}_{[-]}\o{}^{\alpha}) X\t{}_{[-]}\o{}^{\beta} \gamma'(X\t{}_{[-]}\t)(X\o{}_{[+]}X\o{}_{[-]}\o{}^{\alpha'}) X\o{}_{[-]}\o{}^{\beta'} \gamma(X\o{}_{[-]}\t)\\
        =&(X{}_{[+]}X{}_{[-]}\o{}^{\alpha}) X{}_{[-]}\o{}^{\beta}  \gamma(X{}_{[-]}\t)=X_{[+]}X_{[-]}\\
        =&\overline{\varepsilon(X)},
    \end{align*}
    where we use the method of cancelation (as in Theorem \ref{thm. equivalence between cleft extension and twist crossed product})
    in the 2nd step. Recall that
    \[X\la n=X_{[+]}X_{[-]}\o{}^{\alpha}\ot_{\BB}X_{[-]}\o{}^{\beta}n\gamma(X_{[-]}\t).\]
    So on the one hand,
\begin{align*}
    (X\t&\la' n)u(X\o)\\
    =&(X\t{}_{[+]}X\t{}_{[-]}\o{}^{\alpha'}) X\t{}_{[-]}\o{}^{\beta'} n\gamma'(X\t{}_{[-]}\t)(X\o{}_{[+]}X\o{}_{[-]}\o{}^{\alpha'}) X\o{}_{[-]}\o{}^{\beta'} \gamma(X\o{}_{[-]}\t)\\
    =&(X{}_{[+]}X{}_{[-]}\o{}^{\alpha'}) X{}_{[-]}\o{}^{\beta'}  n\gamma(X{}_{[-]}\t)
\end{align*}
By a similar computation
\begin{align*}
    u(X\t) X\o\la n=(X{}_{[+]}X{}_{[-]}\o{}^{\alpha'}) X{}_{[-]}\o{}^{\beta'}  n\gamma(X{}_{[-]}\t).
\end{align*}
So $(X\t\la' n)u(X\o)=u(X\t) X\o\la n$. Recall that
\[\sigma(X, Y)=X_{[+]}Y_{[+]}(Y_{[-]}\o X_{[-]}\o)^{\alpha}\ot_{\BB}(Y_{[-]}\o X_{[-]}\o)^{\beta}\gamma(Y_{[-]}\t)\gamma(X_{[-]}\t).\]
So on the one hand,
\begin{align*}
    u(X\t& Y\t)\sigma(X\o, Y\o)\\
    &=(X\t{}_{[+]}Y\t{}_{[+]}(Y\t{}_{[-]}\o X\t{}_{[-]}\o)^{\alpha'})(Y\t{}_{[-]}\o X\t{}_{[-]}\o)^{\beta'}\gamma(Y\t{}_{[-]}\t X\t{}_{[-]}\t)\\
    &(X\o{}_{[+]}Y\o{}_{[+]}(Y\o{}_{[-]}\o X\o{}_{[-]}\o)^{\alpha})(Y\o{}_{[-]}\o X\o{}_{[-]}\o)^{\beta}\gamma(Y\o{}_{[-]}\t)\gamma(X\o{}_{[-]}\t)\\
    &=X_{[+]}Y_{[+]}(Y_{[-]}\o X_{[-]}\o)^{\alpha'}\ot_{\BB}(Y_{[-]}\o X_{[-]}\o)^{\beta'}\gamma(Y_{[-]}\t)\gamma(X_{[-]}\t),
\end{align*}
    where the 2nd step uses the method of cancelation. On the other hand,
    \begin{align*}
        \sigma'(&X\th, Y\t)u(X\t)(X\o\la u(Y\o))=\sigma'(X\t\t, Y\t)u(X\t\o)(X\o\la u(Y\o))\\
        &=(X\t\t{}_{[+]}Y\t{}_{[+]}(Y\t{}_{[-]}\o X\t\t{}_{[-]}\o)^{\alpha'})(Y\t{}_{[-]}\o X\t\t{}_{[-]}\o)^{\beta'}\gamma'(Y\t{}_{[-]}\t)\gamma'(X\t\t{}_{[-]}\t)\\
        &(X\t\o{}_{[+]}X\t\o{}_{[-]}\o{}^{\alpha'}) X\t\o{}_{[-]}\o{}^{\beta'} \gamma(X\t\o{}_{[-]}\t)\\
        &(X\o{}_{[+]}X\o{}_{[-]}\o{}^{\alpha})X\o{}_{[-]}\o{}^{\beta}((Y\o{}_{[+]}Y\o{}_{[-]}\o{}^{\alpha'}) Y\o{}_{[-]}\o{}^{\beta'} \gamma(Y\o{}_{[-]}\t))\gamma(X\o{}_{[-]}\t)\\
        &=(X\t{}_{[+]}Y\t{}_{[+]}(Y\t{}_{[-]}\o X\t{}_{[-]}\o)^{\alpha'})(Y\t{}_{[-]}\o X\t{}_{[-]}\o)^{\beta'}\gamma'(Y\t{}_{[-]}\t)\gamma(X\t{}_{[-]}\t)\\
        &(X\o{}_{[+]}X\o{}_{[-]}\o{}^{\alpha})X\o{}_{[-]}\o{}^{\beta}((Y\o{}_{[+]}Y\o{}_{[-]}\o{}^{\alpha'}) Y\o{}_{[-]}\o{}^{\beta'} \gamma(Y\o{}_{[-]}\t))\gamma(X\o{}_{[-]}\t)\\
        &=(X_{[+]}Y\t{}_{[+]}(Y\t{}_{[-]}\o X_{[-]}\o)^{\alpha'})(Y\t{}_{[-]}\o X_{[-]}\o)^{\beta'}\gamma'(Y\t{}_{[-]}\t)\\
        &((Y\o{}_{[+]}Y\o{}_{[-]}\o{}^{\alpha'}) Y\o{}_{[-]}\o{}^{\beta'} \gamma(Y\o{}_{[-]}\t))\gamma(X_{[-]}\t)\\
        &=(X_{[+]}Y_{[+]}(Y_{[-]}\o X_{[-]}\o)^{\alpha'})(Y_{[-]}\o X_{[-]}\o)^{\beta'} \gamma(Y_{[-]}\t)\gamma(X_{[-]}\t),
    \end{align*}
    where the 2nd, 3rd and 4th steps use the method of cancelation. As a result, $\sigma'(X, Y)=u(X\fo Y\th)\sigma(X\th, Y\t)(X\t\la u^{-1}(Y\o))u^{-1}(X\o)$ and $(\sigma,\la)$ is equivalent to $(\sigma',\la')$.
\end{proof}

\end{document}